\documentclass{amsart}

\usepackage{tikz-cd}
\usetikzlibrary{calc}
\usetikzlibrary{decorations.markings}

\usepackage{amssymb}

\usepackage{thmtools,thm-restate}
\usepackage{mathtools}

\usepackage{bbm}
\usepackage[hyperfootnotes=false]{hyperref}
\hypersetup{
    colorlinks=true,
    linkcolor=blue,
    urlcolor=blue,
    citecolor=blue,
}

\usepackage[capitalize]{cleveref}

\theoremstyle{plain}
\newtheorem{proposition}{Proposition}[section]
\newtheorem{theorem}[proposition]{Theorem}
\newtheorem{lemma}[proposition]{Lemma}
\newtheorem{corollary}[proposition]{Corollary}

\newtheorem{question}[proposition]{Question}
\newtheorem{claim}[proposition]{Claim}

\theoremstyle{definition}

\newtheorem{definition}[proposition]{Definition}
\theoremstyle{remark}
\newtheorem{remark}[proposition]{Remark}

\newcommand{\R}{\mathbb{R}}
\newcommand{\Pro}{\mathbb{P}}
\newcommand{\Z}{\mathbb{Z}}

\newcommand{\F}{\mathcal{F}}

\newcommand{\Jac}{\operatorname{Jac}}
\newcommand{\Comb}{\operatorname{Comb}}
\newcommand{\comb}{\operatorname{comb}}
\newcommand{\Sing}{\operatorname{Sing}}
\newcommand{\sing}{\operatorname{sing}}
\newcommand{\str}{\operatorname{str}}
\newcommand{\sgn}{\operatorname{sgn}}
\newcommand{\Sect}{\operatorname{Sect}}

\newcommand{\bTheta}{\boldsymbol{\Theta}}
\newcommand{\f}{\varphi}
\newcommand{\e}{\varepsilon}

\newcommand{\op}{\operatorname}
\newcommand{\wt}{\widetilde}
\newcommand{\wh}{\widehat}

\title[Angle cocycles, Euler Characteristic, and simplicial volume]{Angle cocycles, Euler characteristic, and simplicial volume of smooth manifolds}

\author{Michelle Bucher, Stefano Francaviglia and Grazia Rago}

\begin{document}

\begin{abstract} 
A celebrated conjecture of Gromov states that for closed aspherical manifolds, the vanishing of simplicial volume implies the vanishing of the Euler characteristic.
We prove Gromov’s conjecture for closed convex projective manifolds and, more generally, for closed orientable smooth manifolds admitting an immersed straightening. Generalising the straightening of simplices in properly convex domains in real projective space, we introduce the notion of immersed straightening and prove that, for every integer $d\geq 1$, there exists a constant $C_d>0$ such that every closed orientable smooth $d$-manifold $M$ admitting an immersed straightening satisfies $|\chi(M)|\le C_d||M||$.
\end{abstract}

\maketitle
\setcounter{tocdepth}{1}
\tableofcontents

\section{Introduction}

The simplicial volume $\|M\|$ of closed oriented manifolds is a homotopy invariant introduced by Gromov in \cite{Gromov1982} for its impressive applications to the geometry, topology and rigidity properties of manifolds. It remains however a mysterious invariant. In particular, the following conjectural relationship with the Euler characteristic, originally proposed by Gromov, is still open:

\begin{question}{ \cite[p. 232]{GromovAsInvOfInfGps}}\label{question:gromov} 
Does the Euler characteristic of a closed oriented aspherical manifold vanish if its simplicial volume vanishes?
\end{question}

Note that the hypothesis on  asphericity is necessary since spheres $S^d$ have vanishing simplicial volume, but their Euler characteristic is $2$ when $d$ is even. 

\cref{question:gromov} is answered in dimension $2$ where one has the precise relation $-2\chi(\Sigma_g)=\|\Sigma_g\| $ for closed surfaces of genus $g\geq 1$. It has further recently been solved in dimension $4$ for nonpositively curved manifolds $M^4$ by Kim and Wan  as a consequence of the inequality $|\chi(M^4)|\leq 11\cdot \|M^4\|$ which the authors prove in \cite{KimWan}. An inequality between Euler characteristic and simplicial volume has also been established by Sullivan \cite{Sullivan} and improved to $|\chi(M^{2d})|\leq (1/2^{2d})\cdot \|M^{2d}\|$ by Smillie \cite{Smillie} for affine (not necessarily aspherical) $(2d)$-dimensional manifolds. Beside these explicite relations, there are some examples where both invariants are known to either simultaneously vanish or not, and we refer to  \cite{LohMoraschiniRaptis} for a detailed discussion. 

After affine manifolds, it is natural to consider convex real projective manifolds. In this
paper we give a positive answer to \cref{question:gromov} for these manifolds.

\begin{restatable}{theorem}{Comparison}
  \label{thm: comparison theorem}
    Let $\Omega\subset\R\Pro^d$ be a properly convex domain. Then, there exists $0<C_d\leq\frac{d+2}{4}$ such
    that for any $\Gamma\le\operatorname{PGL}(\Omega)$ discrete, torsion-free subgroup of
    projective isometries such that $M\coloneqq\Omega/\Gamma$ is closed, the following holds:
    \begin{equation*}
        |\chi(M)|\le C_d\cdot||M||.
    \end{equation*}
\end{restatable}

Theorem \ref{thm: comparison theorem} was obtained by the first author of this paper as a consequence of the more general Theorem \ref{cor: answer to gromov question} and recently recovered by the last two authors. 

Among convex projective manifolds one finds $4$ infinite families and one exceptional locally
symmetric spaces, which are known to have positive simplicial volume
\cite{bucher2007,LafontSchmidt}.  Irreducible symmetric spaces admitting a convex projective
structure where classified in~\cite{Koecher} (the explicit list can be found in \cite[Theorem 2.1]{BenoistSurvey} or \cite[Table 1]{MitulRago}.)
Note that unfortunately, except in the even dimensional real hyperbolic case, the corresponding locally symmetric spaces all have vanishing Euler characteristic, so that \cref{thm: comparison theorem} cannot be used to recover the positivity of their simplicial volume.

Also it is easy to construct example of convex projective manifolds for which both the Euler
characteristic and the simplicial volume vanish. For example considering two convex projective
manifolds  $M_1=\Omega_1/\Gamma_1$ and $M_2=\Omega_2/\Gamma_2$, then by setting $\Omega$ to
be the joint $\Omega=\Omega_1*\Omega_2$ and $\Gamma=\Gamma_1\times\Gamma_2\times \Z$. The
resulting manifold $M=\Omega/\Gamma$ is a circle bundle, and hence $\chi(M)=0=\| M\|$.

In the context of convex projective manifolds where the universal cover is non-decomposable, one might further expect that both the Euler characteristic (with some additional natural restriction) and the simplicial volume are nonzero:

\begin{question}\label{question:poschar} Do compact even-dimensional, convex real projective manifolds with non-decomposable universal cover that is not globally symmetric have non-zero Euler characteristic? 
\end{question}
\begin{question}\label{question:posvol}
  Do compact, convex real projective manifolds with non-decomposable universal cover have  positive simplicial volume?
\end{question}

Note that in view of the positivity of the simplicial volume of locally symmetric spaces, \cref{thm: comparison theorem} implies that a positive answer to \cref{question:poschar} would imply a positive answer to \cref{question:posvol} in the even dimensional case.

Examples of convex projective $4$-manifolds $M_k$ (with non-decomposable universal cover) with $\chi(M_k)=2k$ for any $k>0$ have just been built in \cite{RioloSeppiSlavich}. These manifolds are constructed by gluing hyperbolic pieces along tori. Their simplicial volume is therefore known to be positive. \cref{thm: comparison theorem} implies further that $|\chi(M_k)|=2k\leq \|M_k\|$ and provides an alternative proof of the positivity of their simplicial volume which does not rely on the somehow technical additivity of the simplicial volume along amenable pieces \cite{BBFIPP}.

Except in the (higher rank) locally symmetric spaces case, convex real projective manifolds with non-decomposable universal cover have rank one in Zimmer's sense \cite{ZimmerRank}. Exploiting the rank one property, \cref{question:posvol} has some further positive answers: It is known that the simplicial volume of a closed oriented aspherical manifold is positive if the fundamental group is word-hyperbolic \cite{mineyevStraightening}, and in particular if the universal cover is a strictly convex divisible domain \cite{benoistCD1}. On the other hand, if the universal cover is not a strictly convex divisible domain, it contains a non-empty family of properly embedded simplices (PES), i.e. subsets that are projectively equivalent to simplices and such that the inclusion map is a proper embedding. In this case, there are known examples where the fundamental group is relatively hyperbolic with respect to the stabilizers of the maximal PES's of dimension at least $2$ \cite{IslamZimmer23, benoistCD4,ChoiLeeMarquis,DGKLM2024}, or relatively hyperbolic with respect to the stabilizers of some maximal properly embedded cones \cite{BlayacViaggi}.
In all these cases, the simplicial volume is positive \cite{franceschini,mineyevRelative}. However, it is not known if the fundamental group of closed rank-one convex projective manifolds is always relatively hyperbolic.

\medskip

The starting point of the proof of our main \cref{thm: comparison theorem} and its generalized variant given in \cref{cor: answer to gromov question} below is the remarkable Poincar\'e Formula for spheres, Euclidean  and hyperbolic spaces. Let $\sigma$ be a non-degenerate singular simplex whose image is the convex hull of its vertices. For any of the faces $F$ of $\sigma$, define the \emph{(solid) angle}  $ \alpha_F(\sigma)\in (0,1]$ as the proportion $\alpha_F(\sigma):=\mathrm{Vol}(\sigma\cap B)/\mathrm{Vol}(B)$, where $B$ is any sufficiently small ball centered at any point  (for example the barycenter) in the interior of the face $F$. Poincar\'e's Formula states that

$$     \Theta(\sigma):=\sgn(\sigma)\sum_{F\text{ face of }\sigma}(-1)^{\dim(F)}\alpha_F(\sigma)=\left\{ 
\begin{array}{ll}
  \frac{2}{\mathrm{Vol}(S^d)} \mathrm{Vol}(\sigma)
   & \mathrm{if \ }X=S^d\\ &  \mathrm{and \ }d\mathrm{ \ is \ even},\\
\frac{(-1)^{d/2}2}{\mathrm{Vol}(S^d)} \mathrm{Vol}(\sigma) & \mathrm{if \ }X=\mathbb{H}^d\\ & \mathrm{and \ }d\mathrm{ \ is \ even},\\
  \ \\
  0&\mathrm{otherwise}.
\end{array}
\right.$$ 

These formulas were originally obtained  in the spherical case. In this case, a short pretty proof exploiting the symmetries and compacticity of the sphere can be found in \cite[Part One, Chapter I]{Hopf}. We present below a modern proof  of this formula in the Euclidean and real hyperbolic case.

The reason behind this formula is that,
as will be seen in the sequel of this paper, the left hand side of the equality is a cocycle representing the Euler class. Hirzebruch's proportionality principle gives us the proportionality constants $\pm 2/\mathrm{Vol}(S^n)$ or $0$ according to the above case distinction. This readily implies that the equality should hold up to a coboundary.  Finally, the fact that both this Euler cocycle and the volume cocycle are equivariant for the full group of isometries implies that the coboundary can also be chosen to be equivariant and thus vanish in degree $\leq n-1$ since any convex $(n-1)$-simplex is contained in a hyperplane which is fixed by a reflexion. 

The fact that the left hand side of the above formula is a cocycle representing the Euler class can be seen geometrically as follows: 
Consider a vector field
      on the $1$-skeleton of the first barycentric subdivision of $\sigma$ given by flowing from barycenters of lower to higher dimensional faces and vanishing on the barycenters. Extend it to a  vector field on $\sigma$ without additional zeroes.\footnote{This is the classical vector field used by Hopf to prove that the index of this (and hence any) vector field on a triangulation of a manifold equals its Euler characteristic. } 
The expression $    \Theta(\sigma)$ should now be thought as the contribution by the simplex
$\sigma$ of this associated vector field. For example in dimension $2$, the vector field has
the flow depicted in \cref{fig: illustration1}
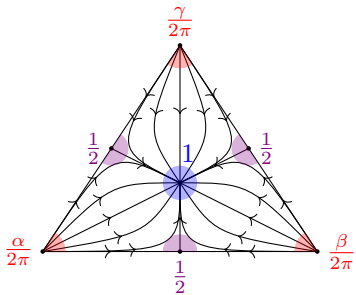
\begin{figure}[h]
  \centering
  \begin{tikzpicture}[x=1ex,y=1ex]
         \coordinate (A) at (0,0);
         \coordinate (C) at (12,18);
         \coordinate (B) at (24,0);
         \coordinate (D) at ($ 1/3*(A) + 1/3*(B)+ 1/3*(C) $);
         \coordinate (E) at ($ 1/2*(A) + 1/2*(B)$);
         \coordinate (F) at ($ 1/2*(A) + 1/2*(C)$);
         \coordinate (G) at ($ 1/2*(C) + 1/2*(B)$);
         \foreach \a in {A,B,C,D,E,F,G}{
           \fill (\a) circle (.2);
         };

    \draw[postaction={decorate,
      decoration={ markings,
        mark=at position 0.5 with {\arrow{>}}}}]
    (A)--(E);
    \draw[postaction={decorate,
      decoration={ markings,
        mark=at position 0.5 with {\arrow{>}}}}]
    (A)--(F);
    
        \draw[postaction={decorate,
      decoration={ markings,
        mark=at position 0.5 with {\arrow{>}}}}]
    (B)--(E);
    \draw[postaction={decorate,
      decoration={ markings,
        mark=at position 0.5 with {\arrow{>}}}}]
    (C)--(F);
    \draw[postaction={decorate,
      decoration={ markings,
        mark=at position 0.5 with {\arrow{>}}}}]
    (B)--(G);
    \draw[postaction={decorate,
      decoration={ markings,
        mark=at position 0.5 with {\arrow{>}}}}]
    (C)--(G);

         \foreach \a in {A,B,C,E,F,G}{
    \draw[postaction={decorate,
      decoration={ markings,
        mark=at position 0.5 with {\arrow{>}}}}]
    (\a)--(D);
    };

   \draw[postaction={decorate,
      decoration={ markings,
        mark=at position 0.5 with {\arrow{>}}}}]
    (A) ..  controls (E) .. (D);
    \draw[postaction={decorate,
      decoration={ markings,
        mark=at position 0.5 with {\arrow{>}}}}]
    (A) ..  controls ($.5*(E)+.5*($.5*(A)+.5*(D)$)$) .. (D);
       \draw[postaction={decorate,
      decoration={ markings,
        mark=at position 0.5 with {\arrow{>}}}}]
    (A) ..  controls (F) .. (D);
    \draw[postaction={decorate,
      decoration={ markings,
        mark=at position 0.5 with {\arrow{>}}}}]
    (A) ..  controls ($.5*(F)+.5*($.5*(A)+.5*(D)$)$) .. (D);
   \draw[postaction={decorate,
      decoration={ markings,
        mark=at position 0.5 with {\arrow{>}}}}]
    (B) ..  controls (E) .. (D);
    \draw[postaction={decorate,
      decoration={ markings,
        mark=at position 0.5 with {\arrow{>}}}}]
    (B) ..  controls ($.5*(E)+.5*($.5*(B)+.5*(D)$)$) .. (D);

       \draw[postaction={decorate,
      decoration={ markings,
        mark=at position 0.5 with {\arrow{>}}}}]
    (B) ..  controls (G) .. (D);
    \draw[postaction={decorate,
      decoration={ markings,
        mark=at position 0.5 with {\arrow{>}}}}]
    (B) ..  controls ($.5*(G)+.5*($.5*(B)+.5*(D)$)$) .. (D);
   \draw[postaction={decorate,
      decoration={ markings,
        mark=at position 0.5 with {\arrow{>}}}}]
    (C) ..  controls (F) .. (D);
    \draw[postaction={decorate,
      decoration={ markings,
        mark=at position 0.5 with {\arrow{>}}}}]
    (C) ..  controls ($.5*(F)+.5*($.5*(C)+.5*(D)$)$) .. (D);

      \draw[postaction={decorate,
      decoration={ markings,
        mark=at position 0.5 with {\arrow{>}}}}]
    (C) ..  controls (G) .. (D);
    \draw[postaction={decorate,
      decoration={ markings,
        mark=at position 0.5 with {\arrow{>}}}}]
    (C) ..  controls ($.5*(G)+.5*($.5*(C)+.5*(D)$)$) .. (D);

    \fill[red, opacity=.3] (A) ++(0:2) arc (0:57:2) -- (A) --cycle;
    \fill[red, opacity=.3] (B) ++(180:2) arc (180:123:2) -- (B) --cycle;
    \fill[red, opacity=.3] (C) ++(237:2) arc (237:304:2) -- (C) --cycle;

    \fill[violet, opacity=.3] (E) ++(180:1.5) arc (180:0:1.5) -- (E) --cycle;
    \fill[violet, opacity=.3] (F) ++(237:1.5) arc (237:417:1.5) -- (F) --cycle;
    \fill[violet, opacity=.3] (G) ++(123:1.5) arc (123:303:1.5) -- (G) --cycle;

    \fill[blue, opacity=.3] (D) circle (1.5);

    \node[red,left] at (A) {$\frac{\alpha}{2\pi}$};
    \node[red,right] at (B) {$\frac{\beta}{2\pi}$};
    \node[red,above] at (C) {$\frac{\gamma}{2\pi}$};

    \node[violet,below] at (E){$\frac{1}{2}$};
    \node[violet,left] at (F){$\frac{1}{2}$};
    \node[violet,right] at (G){$\frac{1}{2}$};

    \node[blue,above right,shift={(-.7,1)}] at (D)  {$1$};
  \end{tikzpicture}
  \caption{The vector field in dimension $2$}\label{fig: illustration1}
\end{figure}
and its contribution is
$$\frac{1}{2\pi}(\alpha+\beta+\gamma-3\cdot \pi+2\pi)=\frac{\alpha+\beta+\gamma-\pi}{2\pi}.$$

The cocycle relation follows from the fact that if $\sigma$ is an $(n+1)$-dimensional simplex, then the pullback of the tangent bundle $TX$ by $\sigma_{\mid \partial \Delta^{n+1}}:\partial \Delta^{n+1}\rightarrow X$ is trivial. Thus the index of any section with isolated zeroes on $\partial \Delta^{n+1}\sim S^{n}$ vanishes. This is in particular the case for the section given by the pullback of the above associated vector field. 

It is even easier to show that it represents the Euler class for any  closed locally $X$ manifold $M$. Indeed the evaluation of the cocycle (which descends to $M$ by equivariance)  on any triangulation gives the index of the associated vector field, which is precisely the Euler characteristic.

Of course, the expression $   \Theta(\sigma)$ cannot be evaluated on arbitrary simplices $\sigma$, but it turns out that the arguments above, viewing  $   \Theta(\sigma)$ as the contribution of $\sigma$ to a vector field, still makes sense for immersed simplices (see \cref{def:smoothsim}). While the motivation is geometric, our proofs will rather be proven from  a combinatorial point of view. The generalized expression of $   \Theta(\sigma)$  which we give in \cref{def:angledefect} is clearly bounded and for manifolds admitting an immersed straightening (see \cref{def:Mais}), induce a bounded cocycle representing the Euler class, which immediatly implies: 

\begin{restatable}[Bucher, unpublished]{theorem}{Inequality}
  \label{cor: answer to gromov question}
  There exists a positive constant $C_d\leq\frac{d+2}{4}$ such that for any closed orientable smooth
  $d$-manifolds admitting an immersed straightening we have
    \begin{equation*}
        |\chi(M)|\le C_d||M||.
      \end{equation*}
\end{restatable}

This result was disseminated by the first author of this paper in 2011 in a series of conferences in the slightly broader context of \emph{immersed simplicial volume} which will not be addressed in this paper. The recent developments \cite{KimWan} on this topic convinced us that it was time to finally write this down. 

The proof of Theorem  \ref{cor: answer to gromov question} will be given at the end of Section \ref{proof of main thm}. 

Observe that spheres of dimension $\geq 2$ do not admit an immersed straightening, but there is
some hope that Gromov's \cref{question:gromov}  could be solved using \cref{cor: answer to
  gromov question} by showing that any aspherical manifold admits an immersed straightening,
possibly weakening our \cref{def:Mais} of immersed straightening.

We remark that, except in constant curvature, any straightening whose $1$-skeleton is given by geodesic segments will not produce immersed straightenings in dimension $\geq 3$ (see \cref{examples}).

It would be interesting to understand whether or not some barycenter method can
produce immersed straightenings, which is unknown  even in negative (non-constant) curvature.

\medskip

\textsc{Acknowledgements.}
We are grateful, for the many, intense and inspiring conversations,  to Roberto Frigerio, Mitul
Islam, Jean-Fran\c cois Lafont, Clara L\"oh, Francesco Milizia, Inkang Kim, Marco Moraschini, Jun
Murakami, Beatrice Pozzetti, Alan Reid, Stefano Riolo,  Andrea Seppi, Leone Slavich, Nicolas
Tholozan, Costantin Vernicos, Gabriele Viaggi.

\medskip

\textsc{Funding.} This research was partially supported by the SNSF, INdAM and GNSAGA, by
the European Union - NextGenerationEU under the National
Recovery and Resilience Plan (PNRR) - Mission 4 Education and research
- Component 2 From research to business - Investment 1.1 Notice Prin
2022 -  DD N. 104 del 2/2/2022, from title Geometry and topology of
manifolds, proposal code 2022NMPLT8
- CUP  J53D23003820001, and European Research Council under the HE Program, grant D.A.T, G.A.
no 101077154, CUP J33C22004270006.

\section{Combinatorial, immersed simplices, and straightening}\label{sec:definitions}

  \begin{definition} For any $n$,
        let $\Delta^n$ be the standard $n$-simplex in $\R^{n+1}$, defined as the convex hull
        of the canonical basis $e_0,\dots,e_n$. For any $I\subsetneq\{0,\dots,n\}$ we refer to
        the face opposite to $I$ ({\em i.e.} the convex hull of $\{e_i, i\notin I\}$)
        as the {\em $I^{th}$ face} of $\Delta^n$, the
        $i^{th}$ {\em facet} being the face opposite to $e_i$.
        For $I=\{i_0,\dots,i_k\}$ we denote
by  $\f_I:\Delta^{k}\to \Delta^n$ the affine map induced by
$(e_0,\dots,e_k)\mapsto (e_{i_0},\dots,e_{i_k})$. We define $\f^I=\f_{I^c}$,
writing $\f^i$ for $\f^{\{i\}}$.
    \end{definition}

\begin{definition}
    Let $X$ be a set. A {\em combinatorial $k$-simplex}
    in $X$ is an ordered tuple $S=(v_0,\dots,v_k)$ of points in $X$, which we
    refer to as its {\em vertices}. We say that $S$ has combinatorial dimension $k$ and we
    write $\dim_c(S)= k$. For any $I\subsetneq\{0,\dots,k\}$, set
    $$
    S^I\coloneqq  (v_i \mid i\not\in I).
    $$
    We refer to $S^I$ as the {\em $I^{th}$-face of $S$}, $S^i$ being the facet opposite to
    $v_i$. (We accept $S^{\emptyset}=S$).
    We say that an $\ell$-simplex $F$ is a {\em face} of $S$ if
    there exists $I\subseteq\{0,\dots,k\}$ such that $F=S^I$ (hence $\dim_c(F)=\ell =k-|I|$).  
  \end{definition}

      \begin{definition}
Let $X$ be any topological space. We denote by $\Comb_k(X)\coloneqq X^{k+1}$ the set of all combinatorial
$k$-simplices in $X$, and by $\Sing_k(X)$ the set of all singular $k$-simplices in $X$. We set 
$$\Comb(X)=\bigcup_{k\geq 0}\Comb_k(X),\qquad \Sing(X)=\bigcup_{k\geq0}\Sing_k(X).$$
We define the map $\comb:\Sing(X)\to\Comb(X)$ by
    associating to  any singular $k$-simplex $\sigma:\Delta^k\to X$ the ordered
    tuple of its vertices, namely
    $$
    \comb(\sigma)\coloneqq (\sigma(e_0),\dots,\sigma(e_k)).$$
  \end{definition}

      \begin{definition}\label{def:smoothsim}
    Let $N$ be a smooth $d$-manifold. A singular simplex $\eta:\Delta^k\to N$ is {\em smooth}
    if it is 
    the restriction of a smooth map $\R^{k+1}\to N$. In particular, the differential of $\eta$
    is well-defined on $T\R^{k+1}$.
    Set $$\eta^I=\eta\circ\f^I\quad \text{ and } \quad \eta^i=\eta\circ\f^i.$$
        We say that $\eta$ is {\em immersed} if $d\eta|_{T\Delta^k}$ has rank $k$
    on any point of $\Delta^k$ (which forces $k\leq d$). In general we say that $\eta$ is {\em
      $d$-immersed} 
    if $d\eta|_{T\Delta^k}$ has maximal rank on any point of $\Delta^k$ and, in case $k>d$, we
    require that  $\eta^I$ is immersed
    for any $I$ such that $k-|I|=d$.

    If $k=d$ and $N$ is oriented, we define
              \begin{equation*}
                \sgn(\eta)\coloneqq
                \begin{cases}
                  +1,\quad&\text{ if $\eta$ is immersed and orientation preserving},\\
                  -1,\quad&\text{ if $\eta$ is immersed and orientation reversing},\\ 
                  0,\quad&\text{ if $\eta$ is not immersed}.
                \end{cases}
              \end{equation*}
    \end{definition}

    \begin{definition}
      Let $N$ be a smooth $d$-manifold. A singular simplex $\sigma:\Delta^k\to N$, continuous
      but not necessarily smooth, is said to have {\em smooth image} if there exists
      an orientation preserving homeomorphism $\f:\Delta^k\to \Delta^k$, compatible with
      faces (that is for any $I$, $\f\circ\f^I$ and $\f^I$ have the same image) and so that
      $\sigma\circ\f$ is smooth.
      We call such a $\f$ a smooth reparametrization of $\sigma$. We say that
      $\sigma$ is immersed or $d$-immersed if $\sigma\circ \f$ is, and we define the sign of
      $\sigma$ as the sign of $\sigma\circ\f$. 
    \end{definition}
    
In general, there is no a natural way to define a singular simplex, given only the set of
vertices. Given a map $\sing:\Comb(M)\to \Sing(M)$, the composition $\sing\circ \comb$ is what
is frequently referred to as a {\em straightening}. Of course, not any map is useful. We will need straightenings with the following additional properties:

\begin{definition}\label{def:Mais}
      Let  $M$ be a smooth $d$-dimensional manifold and let $\Omega=\wt M$ be its universal
      cover.  We 
      denote by $\Gamma\coloneqq\pi_1(M)$ the group of deck-transformations of $\pi:\Omega\to
      M$. We say that $M$ {\em admits an immersed straightening} if there exists 
    a map $$\sing:\Comb(\Omega)\to \Sing(\Omega)$$ with the following properties:
    \begin{enumerate}
    \item\label{pointcombsing} The map $\sing$ is the identity on vertices $$\comb\circ \sing=Id.$$
    \item\label{pointsmooth} The map $\sing$ is graded and with smooth image:
      for any combinatorial $k$-simplex $S\in\Comb(\Omega)$, $\sing(S):\Delta^k\to \Omega$
      is a singular $k$-simplex with smooth image.
    \item\label{pointfaces} The map $\sing$ is compatible with faces:
for any $S\in\Comb_k(\Omega)$ and $I\subseteq\{0,\dots,k+1\}$
$$\sing(S^I)=\sing(S)\circ\f^I.$$
    \item\label{pointequivariant}
      $\sing$ is equivariant with respect to the action of $\Gamma$ on $\Omega$: 
      $$\sing(\gamma S)=\gamma\sing(S)$$
      for any $S\in\Comb(\Omega)$ and $\gamma\in\Gamma$.
    \item\label{pointnond} 
 For any $k$ the set $\mathcal D_k$ of
 $k$-simplices $S$ such that $\sing(S)$ is not $d$-immersed, has {\em measure zero} in
 $\Omega^{k+1}$ with respect to the   Lebesgue measure class.
    \item\label{pointeqhom}
      The {\em straightening} associated to $\sing$, defined as
      $$\str:\Sing(\Omega)\to\Sing(\Omega) \qquad
      \str\coloneqq\sing\circ\comb,$$
      is required to be equivariantly homotopic to the identity, relative to
      vertices, via an homotopy compatible with faces.  That is to say, there are maps
      $H_k:[0,1]\times \Sing_k(\Omega)\to \Sing_k(\Omega)$ such that for any $k$,
      $H_k(t,\sigma)$ is a  
      $\Gamma$-equivariant homotopy between $\sigma$ and $\str(\sigma)$ that
      fixes the vertices of $\sigma$ and such that $H_{k-|I|}(t,\sigma\circ \f^I)$ is the
      restriction of $H_k(t,\sigma)$ to the image of $\f^I$. 
      We denote by $\str$ also the map on $M$ defined by $\str(\sigma)=\pi(\str(\wt\sigma))$
      (where $\wt\sigma$ is any lift of $\sigma$).  
    \item \label{pointlem:postriangopen}
      The map $\sing$ is $C^1$-continuous: for any sequence $S_n\to S$ of combinatorial simplices,
      there exists smooth reparametrisations $\sigma$ of $\sing(S)$ and $\sigma_n$ of
      $\sing(S_n)$ such that  
      $$S_n\to S\quad\Longrightarrow\quad \op{dist}_{C^1}(\sigma_n,\sigma)\to 0.$$
      If an auxiliary $\Gamma$-invariant metric on $\Omega$ is fixed,  this is equivalent to 
      requiring that (for the induced $\sup$-metric on $\Omega^d$) for any $\e >0$, there  
      is $\delta>0$ such that for any $S_1,S_2\in\Comb(\Omega)$ with $d(S_1,S_2)<\delta$ there
      are smooth reparametrisations $\sigma_1$ and $\sigma_2$ so that
      $\op{dist}_{C^1}(\sigma(S_1),\sigma(S_2))<\e$.
    \item \label{pointprop:posortri}
     For any orientation of $\Omega$, we require that the map $\str$ locally preserves
     orientation on simplices. More precisely, there exists $t>0$ such that for any smooth, immersed
     $d$-simplex $\sigma$, for any sufficiently fine $t$-thick\footnote{
    An Euclidean $k$-simplex is an affine parametrisation of $k+1$ points in $\R^n$. 
    The thickness of an Euclidean simplex $s$ is defined as follows
    \cite[Definition 9.2]{munkres}: The radius $r(s)$ is the distance from the barycenter
    to $\partial s$  
    and the thickness of $s$ is the ratio $r(s)/\op{diam}(s)$. By \cite[Theorem
    9.4]{munkres}, any $\Delta^d$ has arbitrarily fine
    subdivisions whose simplices have thickness bounded below.} subdivision  of
      $\Delta^d$, and for any simplex $s$ of that subdivision
          $$\sgn(\str(\eta_s))=\sgn(\sigma),$$ where $\eta_s:\Delta^d\to s$ is any
      orientation preserving affine parametrisation of $s$. 
    \end{enumerate}
\end{definition}

Our main example of manifolds admitting immersed straightenings will be convex real projective manifolds (\cref{thm cvx imm}).

    \begin{remark} Condition $(\ref{pointprop:posortri})$ can be restated by saying that smooth
      reparametrisations of $\str$
      does not change the sign of small simplices with almost constant differential. 
      More precisely, a differential condition that implies condition
      $(\ref{pointprop:posortri})$ is the 
      following. Embed 
      $M$ in some $\mathbb{R}^N$ so that now for any smooth $d$-simplex $\sigma$, we can identify
      $d\sigma$ with 
      a $d\times N$ matrix. Then we require that 
      for all $C_2>C_1>0$ and for any $d\times N$ matrix $A$ with rank $d$, there exists
      $\e_1>0$ such 
      that $\forall \  \e_0>0$ there is $0<\e<\e_0$ such that if a smooth reparametrisation
      $\sigma$ of $\eta$  satisfies
      \begin{itemize}
      \item $\sigma$ is $(\e C_1,\e C_2)$-biLipschitz,
      \item $\forall \ x\in\Delta^d$, $|\frac{d\sigma_x}{\e}-A|<\e_1$.
      \end{itemize}
      Then $\sgn(\str\eta)=\sgn(\eta)$.
This can be easily checked because any immersed simplex is bi-Lispchitz and thickness implies
Lispchitz condition on fine subdivisions. For sufficiently fine subdivision, and for any
simplex $s$ of that subdivision, $A$ can be chosen
to be $d\sigma_{x_0}$ for any choice of $x_0\in s$.
    \end{remark}

    \section{The definition of the cocycle}\label{defsect3}
    \subsection{Angles as sectors.} 
    Standard Euclidean angles in $\R^2$ are just normalized length of circular sectors.  
    Our approach is to define angles in higher dimensions in terms of measures of sectors of spaces.
    
        \begin{definition} \label{def:planeandsectors}
          Let $\Delta^n$ be the standard $n$-simplex in $\R^{n+1}$. We define
        \begin{itemize}
        \item $\Pi^k$ to be the supporting  space of the $k^{th}$ facet of $\Delta^n$, {\em
            i.e.} the   
       $(n-1)$-subspace of $\R^{n+1}$ defined by the equations $\sum_{i}x_i=1, \ x_k=0$;
     \item  $\Pi^k_+$ to be the $n$-half hyperspace defined by $\sum_ix_i=1,\  x_k>0$.
     \end{itemize}

     For $I\subsetneq\{0,\dots,n\}$ define the $I^{th}$ sector as (see \cref{fig:settore})
     $$\Sect^I\coloneqq \cap_{i\in I}\Pi_+^i.$$
    By convention  we set $\Sect^\emptyset\coloneqq\{\sum_ix_i=1\}$ for the supporting
    hyperplane of $\Delta^n$.
    \end{definition}

\begin{figure}[h]
  \centering

 \begin{tikzpicture}[x=.8ex,y=.8ex]

   \draw[very thin] (0,0)--(20,30)--(40,20)--(20,-10)--cycle;
   \draw[very thin] (20,30)--(40,30)--(34,21);
   \draw[very thin,dashed] (20,30)--(40,30)--(20,0)--(0,0);
   \fill[red, opacity=.25] (0,0)--(20,-10)--(20,0)--cycle;
   \fill[cyan!10, opacity=.5] (0,0)--(20,30)--(40,30)--(20,0)--cycle;
   \fill[cyan!25, opacity=.5] (0,0)--(20,30)--(40,20)--(20,-10)--cycle;
   \fill[red, opacity=.25] (20,30)--(40,20)--(40,30)--cycle;

   \coordinate (A) at (4,6);
   \coordinate (B) at (14,21);
   \coordinate (C) at (20,-2);
   \coordinate (D) at (23,6);

   \fill (A) circle (.2);
   \fill (B) circle (.2);
   \fill (C) circle (.2);
   \fill (D) circle (.2);

   \draw[very thin] (A) -- (C) --(B) --(D) -- (C);
   \draw[very thin, dashed] (A) --(D);

   \draw[thick, blue] (A)--(B);

   \node[blue] at (7,15) {$S^I$};
   \node at (50,20) {$\Sect^I$};
 \end{tikzpicture}
    \caption{Sector of a face of $\Delta^3$.}
\label{fig:settore}
\end{figure}
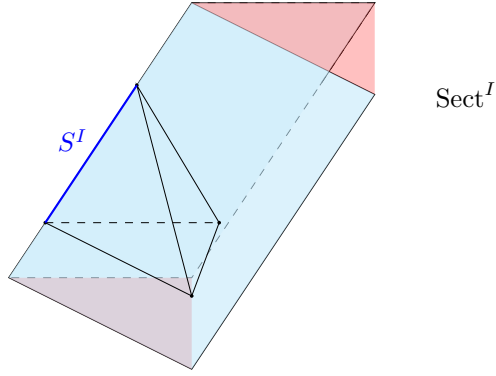

    \begin{definition}
         Let $N$ be any smooth $d$-manifold.
         Let $\eta:\Delta^d\to N$ be a simplex with smooth immersed image and let $\sigma$ be any
         smooth reparametrisation of $\eta$.
 For any $I\subsetneq\{0,\dots,d\}$ and for any $v$ in the $I^{th}$ face of $\Delta^d$ (that is to
 say $v\notin I$),  
    the $I$-{\em sector} of $\eta$ at the $v^{th}$ vertex, is the subset of
  $T_{\eta(e_v)}(N)$ given by  

    $$
    \Sect^I_v(\eta)\coloneqq d\sigma_{e_v}(\Sect^I).
    $$
By convention we have   $\Sect^\emptyset_v{\eta}= T_{\eta(e_v)}N$.
    \end{definition}

It is clear that $\Sect_v^I(\eta)$ does not depend on the chosen reparametrisation $\sigma$.
    Now, we want to measure sectors.

    \begin{lemma}\label{lemma:limitexist}
      Let $g:\R^n\to \R$ be any continuous positive function.
      Let $\op{Vol}$ be the Euclidean volume in $\R^n$ and let $\op{Vol}_g$ the
      volume given by $\op{Vol}(A)=\int_Ag(x)dx$. 
      Let $||\cdot||$ be any norm on $\R^n$, and let $B_r=\{V\in\R^n: ||V||<r\}$.
      Let $\Sect\subseteq\R^n$ be any subset defined by linear inequalities.
      Then
      $$\lim_{\e\to0}\frac{\op{Vol}_g(B_\e\cap\Sect)}{\op{Vol}_g(B_\e)}
=\lim_{\e\to0}\frac{\op{Vol}(B_\e\cap\Sect)}{\op{Vol}(B_\e)}=\frac{\op{Vol}(B_1\cap\Sect)}{\op{Vol}(B_1)},$$
and the latter limit exists.
\end{lemma}
\begin{proof} The last claim and the last equality easily follow because $\operatorname{Sect}$ is defined by linear inequalities. 
  Let $A_\e=B_\e\cap\Sect$. Since $g$ is
  continuous and strictly positive, for any $\delta>0$, and in particular for any
  $0<\delta<g(0)$, there is $\e>0$ so that
  $|g(x)-g(0)|<\delta$ for any $x\in B_\e$. Therefore 
   $$
\frac{(g(0)-\delta)\op{Vol}(A_\e)}{(g(0)+\delta)\op{Vol}(B_\e)}\leq
 \frac{\int_{A_\e}g(x)dx}{\int_{B_\e}g(x)dx}
\leq \frac{(g(0)+\delta)\op{Vol}(A_\e)}{(g(0)-\delta)\op{Vol}(B_\e)}.
  $$ 
These inequalities pass to the limit as $\e\to 0$. Letting $\delta\to 0$ proves the lemma. 
\end{proof}

Recall that a Finsler structure on a manifold $M$ is a family of norms $||\cdot||_x$ on any
$T_xM$ that vary continuously on $x$. For example, a Riemannian metric induces  a Finsler
structure. In particular, since any smooth manifold can be given  a Riemannian metric, any
smooth manifold admits a Finsler structure.

\begin{definition}\label{def:muf}
  Let $M$ be a manifold equipped with a Finsler structure $\F=\{||\cdot||_x\}_{ x\in M}$.
  Let $\Sect\subseteq T_xM$ be any
  subset defined by linear inequalities and let  $B_r=\{V\in T_xM: ||V||_x<r\}$. We define
  $$\mu_\F(\Sect)=\frac{\op{Vol}(B_1\cap\Sect)}{\op{Vol}(B_1)}$$
where the volume is computed in any local chart.
\end{definition}

\begin{remark}
By \cref{lemma:limitexist} the definition of $\mu_\F(\Sect)$ does not depend on the chosen local
chart. Moreover, since the Finsler norms are continuous in $x$, the evaluation $\mu_\F(\Sect)$ is continuous
when $\Sect$ varies continuously among subsets in $TM$ defined by linear inequalities.

The measure $\mu_\F(\Sect)$ depends on $\F$ to the extent that $\F$ determines the shapes of the 
sets $B_\e$.
\end{remark}

\subsection{The angle defect.}

\begin{definition}\label{def:angles}
  Let $M$ be a smooth $d$-manifold equipped with a Finsler structure $\F$.
  Let $\eta:\Delta^d\to M$ be $d$-simplex with smooth immersed image.
  For all $I\subset\{0,\dots,d\}$ and $v\in I^c$, we define the 
  {\em angle induced by $\F$ around the $I^{th}$ face,  localized at the $v^{th}$ vertex} as 

  \begin{equation*}
    \theta^I_v(\eta)\coloneqq \mu_\F(\Sect^I_{v}(\eta)).    
\end{equation*}
The {\em top angle} around the $I^{th}$  face is defined as 
    \begin{equation}\label{theta max}
    \theta^I(\eta)=\theta^I_{\max I^c}(\eta),
\end{equation}
where $\max I^c$ is the maximal vertex of the $I^{th}$  face for the natural order $0<\dots<d$.  
\end{definition}

\begin{remark}  \label{rem average}
  The angle around  a face could be defined by averaging
  angles at points (even not necessarily vertices) with respect to any measure on that face and all our statements would hold likewise upon some minor adaptations. The top angle  $\theta^I(\eta)$ corresponds to the Dirac measure concentrated on the vertex with maximal order. If instead we had considered the uniform measure concentrated on vertices, we would have obtained 
    \begin{equation}\label{theta sym}
    \bar\theta^I(\eta)\coloneqq\dfrac{1}{|I^c|}\sum_{v\in I^c}\theta^I_v(\eta).
\end{equation}
It is readily checked that in fact $\bar\theta^I(\eta)$ is the average of $\theta^I(\eta)$ among all possible orderings of the vertices of $\eta$. Beside being more symmetric it has one advantage over the unsymmetrised variant $\theta^I(\eta)$ which is that the corresponding cocycle obviously vanishes in odd dimensions. We will discuss this further at the end of this section. 
    
Another natural choice would be to measure the angle of each   $I^{th}$  face on its barycenter. This would correspond to the contribution of the vector field described in the introduction. 

Finally observe that in cases where angles are constant along faces, as for example in constant curvature or in dimension $2$, all
  definitions coincides.

\end{remark}

\begin{definition}\label{def:angledefect}
    Let $M$ be a smooth oriented $d$-manifold. Let $\F$ be a Finsler structure on $M$ and
    consider the angles induced by $\F$ as in \cref{def:angles}. Let $\eta$ be a $d$-simplex with smooth image. We
    define the {\em angle defect} of $\eta$ as
  $$
      \Theta(\eta) \coloneqq\sgn(\eta)
      \left(\sum_{\ell=0}^{d}(-1)^{d-\ell}
        \sum_{\substack{I\subset\{0,\dots,d\}\\|I|=\ell}}
        \theta^I(\eta)\right),$$
      which includes the convention that $\Theta(\eta)=0$ for non-immersed simplices (which form
      a negligible set by \cref{pointnond}  of \cref{def:Mais}).
    The {\em angle cochain} induced by $\F$ on $M$ is
    the linear map $\bTheta$ defined on the complex generated 
    by singular $d$-simplices with smooth image by linear extension of $\Theta(\eta)$:
    $$\bTheta\left(\sum_ia_i\eta_i\right)=\sum_ia_i\Theta(\eta_i).$$
 \end{definition}

For our proofs and in order to obtain a better bound on the values of $|     \Theta(\eta)|$ (see \cref{cor:bound}), it will be convenient to have an alternative expression for $     \Theta(\eta)$ as in \cref{lemma:formulacoc} below. To do so, we need the following definition (recall that  $\Pi^i$ and $\Pi^i_+$ are as in \cref{def:planeandsectors}).

         \begin{definition} 
           Let $\Delta^n$ be the standard $n$-simplex in $\R^{n+1}$. For any $0\leq k\leq n$, define
           \begin{itemize}
           \item $\Pi^k_-$ to be the $n$-half hyperplane defined by
             $\sum_ix_i=0, x_k\leq0$. It is the complement of $\Pi^k_+$ in $\Sect^\emptyset$.
           \item The \emph{exiting cone} $E_k$ as the cone spanned by the vectors
     $(i-k)(e_i-e_k)$, which is given in terms of half hyperplanes as
     $$ E_k=\left(\bigcap_{i<k} \Pi_-^i\right)\cap \left(\bigcap_{i>k}\Pi_+^i\right).$$
           \end{itemize}
       \end{definition}

For example, the exiting cones of a $2$-simplex are illustrated in \cref{fig:ec}.

     \begin{figure}[h]
       \centering
       \begin{tikzpicture}[x=1ex,y=1ex]
         \coordinate (A) at (0,0);
         \coordinate (B) at (6,9);
         \coordinate (C) at (12,0);
   \fill (A) circle (.2);
   \fill (B) circle (.2);
   \fill (C) circle (.2);

   \draw (A)--(B)--(C)--cycle;
   \draw[dashed] (B) -- (8,12);
   \draw[dashed] (C) -- (16,0);
   \draw[dashed] (B) -- (4,12);

   \fill[red, opacity=.3] (A) --(3,0) arc (0:56:3) --cycle;
   \fill[red, opacity=.3] (C) ++(0:3) arc (0:124:3) -- (C) --cycle;
   \fill[red, opacity=.3] (B)--($(B)+(56:3)$) arc (56:124:3) --cycle;

   \node[left] at (A) {$e_0$};
   \node[below] at (C) {$e_1$};
   \node[right] at (B) {$e_2$};
       \end{tikzpicture}
       \caption{Exiting cones in a triangle}
       \label{fig:ec}
     \end{figure}
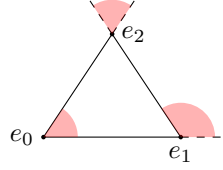

    \begin{lemma}\label{lemma:formulacoc} Let $\sigma$ be a $d$-simplex with smooth image, and
      let $\eta$ be any smooth reparametrisation of our $\sigma$.  Then 
      $$\Theta(\sigma)=\Theta(\eta)=\sgn(\eta)\sum_{k=0}^d(-1)^k\mu_\F(d\eta_{e_k}(E_k)).$$
      \end{lemma}

With this Lemma at hand, we have the following geometric interpretation of $\Theta(\sigma)$: consider a vector field
      on the $1$-skeleton of $\Delta^n$ given by the ordering of the vertices, which has zeroes on all vertices, and on open $1$-simplices, flows from the lower to the higher vertex. 

      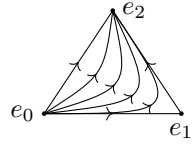
\begin{figure}[h]
        \centering
        \begin{tikzpicture}[x=1ex,y=1ex]
         \coordinate (A) at (0,0);
         \coordinate (B) at (6,9);
         \coordinate (C) at (12,0);
   \fill (A) circle (.2);
   \fill (B) circle (.2);
   \fill (C) circle (.2);

    \draw[postaction={decorate,
      decoration={ markings,
        mark=at position 0.5 with {\arrow{>}}}}]
    (A)--(C);
        \draw[postaction={decorate,
      decoration={ markings,
        mark=at position 0.5 with {\arrow{>}}}}]
    (C)--(B);
    \draw[postaction={decorate,
      decoration={ markings,
        mark=at position 0.5 with {\arrow{>}}}}]
    (A)--(B);

   \draw[postaction={decorate,
      decoration={ markings,
        mark=at position 0.5 with {\arrow{>}}}}]
    (A) ..  controls (C) .. (B);
    \draw[postaction={decorate,
      decoration={ markings,
        mark=at position 0.5 with {\arrow{>}}}}]
    (A) ..  controls ($.5*(C)+.5*($.5*(A)+.5*(B)$)$) .. (B);
    \draw[postaction={decorate,
      decoration={ markings,
        mark=at position 0.5 with {\arrow{>}}}}]
    (A) ..  controls ($.25*(C)+.75*($.5*(A)+.5*(B)$)$) .. (B);
   \draw[postaction={decorate,
      decoration={ markings,
        mark=at position 0.5 with {\arrow{>}}}}]
    (A) ..  controls ($.75*(C)+.25*($.5*(A)+.5*(B)$)$) .. (B);
   
   \node[left] at (A) {$e_0$};
   \node[below] at (C) {$e_1$};
   \node[right] at (B) {$e_2$};
          
        \end{tikzpicture}
        \caption{The vector field $V$}
      \end{figure}
Extend this vector field $V$ to a nowhere vanishing vector field on the interior of $\Delta^n$ so that $V$ has zeros
      only at vertices. The measure of exiting cones is the partial
      contribution in $\Delta^n$ to the index of the zeros of $V$.

    \begin{proof}[Proof of \cref{lemma:formulacoc}.]
The first equality follows jut by definition. Let $\eta$ be any immersed smooth $d$-simplex. We have
      \begin{align*}
        \Theta(\eta)\sgn(\eta)
        &=\sum_{I\subset\{0,\dots d\}}(-1)^{d-|I|}\theta^I(\eta)=
                      \sum_{I\subset\{0,\dots d\}}(-1)^{d-|I|}\theta^I_{\max I^c}(\eta)\\
        &=\sum_{k=0}^d\sum_{\substack{I\subset\{0,\dots d\}\\\max I^c=k}}
        (-1)^{d-|I|}\theta^I_k(\eta)
        =\sum_{k=0}^d\sum_{\substack{I\subset\{0,\dots d\}\\\max I^c=k}}
        (-1)^{d-|I|}\mu_\F(d\eta_{e_k}(\Sect^I)).
      \end{align*}

      Therefore it suffices to prove that for any $k=0,\dots d$
      \begin{equation}
        \label{eq:excones}
        \sum_{\substack{I\subset\{0,\dots d\}\\\max I^c=k}}
        (-1)^{d-|I|}\mu_\F(d\eta_{e_k}(\Sect^I))=(-1)^k\mu_\F(d\eta_{e_k}(E_k))
      \end{equation}
      Fix $k$. Denote by $A=\cap_{i>k}\Pi^i_+$, and by $T=\Sect^\emptyset=T_{e_k}(\Delta^d)$.
      So $\Pi^i_-$ is the
      complement on $T$ of $\Pi_+^i$ and we write $\Pi^i_-=(\Pi_+^i)^c$. Then we have
            $$E_k=\big(\bigcap_{i<k}\Pi_-^i\big)\cap A=
      \big(\bigcap_{i<k}(\Pi_+^i)^c\big)\cap A=\big(\bigcup_{i<k}\Pi_+^i\big)^c\cap A=
      A\setminus \big(\bigcup_{i<k}\Pi_+^i\cap A\big).$$
Passing to measures, since $d\eta_{e_k}$ is invertible, we get by the inclusion-exclusion Principle that
      \begin{align*}
        \mu_\F(d\eta_{\e_k}(E_k))
&=\mu_\F(d\eta_{e_k}(A))-\mu_F(d\eta_{e_k}\big(\bigcup_{i<k}\Pi_+^i\cap A\big))\\
&=\mu_\F(d\eta_{e_k}(A))-\sum_{\emptyset\neq J\subseteq\{0,\dots,k-1\}}
  (-1)^{|J|+1}\mu_F(d\eta_{e_k}\big(\bigcap_{j\in J}\Pi_+^j\cap A\big))\\
&=
  \mu_F(d\eta_{e_k}(A))+\sum_{\emptyset\neq J\subseteq\{0,\dots,k-1\}}
  (-1)^{|J|}\mu_F(d\eta_{e_k}(\big(\Sect^J\cap A\big)).\\
      \end{align*}
      Now we observe that $\Sect^J\cap A=\Sect^{J\cup\{k+1,\dots,d\}}$ and
      $A=\Sect^{\{k+1,\dots,d\}}$ corresponds to the case $J=\emptyset$. Such sets of indices $I$
      are exactly all sets such that $\max I^c=k$. Clearly $|I|=|J|+d-k$.

      So we have
      \begin{align*}
       (-1)^k\mu_\F(d\eta_{\e_k}(E_k))
        &
          =(-1)^k\sum_{\substack{I\subset\{0,\dots,d\}\\\max I^c=k}}
        (-1)^{|I|-d+k}\mu_\F(d\eta_{e_k}(\Sect^I))\\
       &= 
          \sum_{\substack{I\subset\{0,\dots,d\}\\\max I^c=k}}
        (-1)^{d-|I|}\mu_\F(d\eta_{e_k}(\Sect^I)),
      \end{align*}
      which finishes the proof of the lemma.
     \end{proof}

     \begin{corollary}\label{cor:bound}
       If $d=2k$, then for any $d$-simplex $\sigma$ we have $$|\Theta(\sigma)|<\frac{k+1}{2}=\frac{d+2}{4}.$$
     \end{corollary}
     \begin{proof}
       By \cref{lemma:formulacoc} $\bTheta(\sigma)$ is an alternating sum of $d+1=2k+1$ terms, each
       one bounded in absolute value by $1/2$ because $d\sigma_v(E_k)$ is the image of a convex
       cone under a linear map. Then we have $k$ terms of one sign, and $k+1$ of the
       other. Therefore the difference
       between positive and negative contributions is bounded 
       in absolute value by $\frac{k+1}{2}$.
     \end{proof}

\subsection{The symmetrized angle defect} \label{sym variant} Replacing $ \theta^I(\eta)$ from Equation (\ref{theta max}) by $\bar\theta^I(\eta)$ as in Equation (\ref{theta sym}) in the \cref{def:angledefect} of $\bTheta$, we obtain a cochain $\bar\bTheta$ which is by construction the symmetrization of $\bTheta$.
As a consequence, all results proven for    $\bTheta$ immediately hold for $\bar\bTheta$.

Observe that \cref{lemma:formulacoc}, by symmetrisation, implies that there is a
certain union of sectors, counted with multiplicity, so that $\bar\bTheta(\eta)$ is the sum, over
all vertices, of contributions of such sectors. For example, in a triangle, at every vertex
there are four sectors which will be couted with the following multiplicities while we average
over the $6$ permutations of the vertices:

\begin{center}
  \begin{tikzpicture}[x=.5ex,y=.5ex]
    \draw (-14,0)--(14,0);
    \draw[dashed] (14,0)--(40,0)--(20,20)--(10,10);
    \draw (-10,-10)--(10,10);
    \node at (6,3){\scriptsize $+2$};
    \node at (-7,-3){\scriptsize $+2$};
    \node at (5,-3){\scriptsize $-1$};
    \node at (-5,3){\scriptsize $-1$};

    \node at (55,0){$= -1+$};
    \begin{scope}[shift={(80,0)}]
          \draw (-14,0)--(14,0);
    \draw (-10,-10)--(10,10);
    \node at (6,3){\scriptsize $+3$};
    \node at (-7,-3){\scriptsize $+3$};
    \node at (4,-3){\scriptsize $0$};
    \node at (-4,3){\scriptsize $0$};

    \end{scope}
  \end{tikzpicture}
\end{center}
The contribution to $\bTheta$ of every vertex of angle $\alpha$ of the triangle is thus $\frac{1}{6}+\frac{\alpha}{2\pi }$. (Here we used that
unit balls of the Finsler norms are symmetric with respect to the origin.) Summing over all vertices we recover the familiar expression $(1/(2\pi))(-\pi+\alpha+\beta+\gamma)$.  Similar (more
complicated) formula holds in general.

By using \cref{lemma:formulacoc}, we can immediately deduce that $\bar\bTheta$
vanishes   in odd dimension. (This is not the case for $\bTheta$ in general.) To see this,
observe that, at every vertex, each sector appears with its symmetric with respect to the origin. 
Since angles induced by Finsler metric are symmetric because unit
balls of Finlser norms are symmetric, they have the same measure.  But their contribution differs by a sign of $(-1)^d$, so that they all cancel out when $d$ is odd. 
 \section{The cocycle relation}
In this section we prove the cocycle condition $\bTheta(\partial \sigma)=0$. We start by
setting some notation. 

In a $d$-manifold $M$, if $\sigma:\Delta^{d+1}\to M$ is a smooth $(d+1)$-simplex, then
$\partial\sigma=\sum_i(-1)^i\sigma^i$  (where $\sigma^i:\Delta^d\to M$
are the simplices  defined in \cref{def:smoothsim}).

    Technically $\theta^I_v(\sigma^i)$ is defined when $I$ is a subset of $\{0,\dots,d\}$, and
    not for subsets of $\{0,\dots,d+1\}$.
    If we denote by $\iota_i:\{0,\dots,d\}\to \{0,\dots, d+1\}$ the map $$\iota_i(j)=
    \begin{cases}
      j & \text{ if } j<i\\
      j+1 & \text{ otherwise} 
    \end{cases}
    $$  then
    (see the beginning of \cref{sec:definitions} for definitions), the angle that we see at the
    face opposite to $I$, at the 
    $v^{th}$ vertex, in $\sigma^i=\sigma\circ\f^i$, is
    $$\theta^{\iota_i^{-1}(I)}_{\iota_i^{-1}(v)}(\sigma^i).$$
And similarly for sectors.
    To avoid cumbersome notations,
    for any $I\subset\{0,\dots,d+1\}$ containing $i$, and for any $v\notin I$ we set
    $$\Sect_v^I(\sigma^i)\coloneq\Sect^{\iota_i^{-1}(I)}_{\iota_i^{-1}(v)}(\sigma^i) =d\sigma_{e_v}(d\f^i_{_{\iota_i^{-1}(v)}}(\Sect^{\iota_i^{-1}(I)}_{\iota_i^{-1}(v)}))
    $$
    $$\theta^I_v(\sigma^i)\coloneq\theta^{\iota_i^{-1}(I)}_{\iota_i^{-1}(v)}(\sigma^i).
   $$
Note that since $\iota_i$ is monotone on $\{0,\dots,d\}$, for any $I\subset\{0,\dots,d+1\}$
containing $i$ we have
$\max(\iota_i^{-1}(I^c))=\iota_i^{-1}(\max I^c)$. We define
$$\theta^I(\sigma^i):=\theta^{\iota_i^{-1}(I)}_{\max\iota_i^{-1}(I^c)}(\sigma^i)
=\theta^{\iota_i^{-1}(I)}_{\iota_i^{-1}(\max I^c)}(\sigma^i).$$

Before going on, we need a little linear algebra lemma.
\begin{lemma}\label{lemma:linearalgebra}
  Let $v_0,\dots,v_d$ be $d+1$ linearly independent vectors of $\R^{d+1}$.
  Let $D:\R^{d+1}\to \R^d$ be a linear map with $1$-dimensional kernel generated by
  $\xi=\sum_i\lambda_iv_i$. Suppose that $\lambda_i\neq 0$ for any $i$.
  Let $F^i$ be the positive cone generated by $v_0,\dots,\wh v_i,\dots,v_d$, and
  define\footnote{Note that $\cup_iF^i$ is the boundary of the positive cone generated by $v_0,\dots,v_d$.}
  $$F^+=\bigcup_{i:\lambda_i>0}F^i\qquad F^-=\bigcup_{i:\lambda_i<0}F^i.$$
  Then
  \begin{enumerate}
  \item $D|_{F^\pm}:F^\pm\to D(F^\pm)$ is a homeomorphism.
  \item If $F^+$ and $F^-$ are both non-empty, then $D(F^+)=D(F^-)$.
  \end{enumerate}
\end{lemma}
\begin{proof}
  The proof is an easy linear algebra exercise and it is left to the reader.
\end{proof}

\begin{theorem}\label{thm:cocyclecondition}
  Let $M$ be an oriented $d$-manifold equipped with a Finsler structure ({\em e.g.} a
  Riemannian manifold). Let $\sigma:\Delta^{d+1}\to M$ be a $(d+1)$-simplex with smooth image
  and $d$-immersed.
  Then 
    $$
    \bTheta(\partial\sigma)=0
    $$
    Where $\bTheta$ is the angle cochain induced by the Finsler structure.
\end{theorem}
\begin{proof}
  Since $\sigma$ has smooth image, up to reparametrisation we can assume that $\sigma$ is
  smooth.
  
  We denote by $F^i$  the $i^{th}$ facet of $\Delta^{d+1}$, oriented as a cell of
  $\partial\Delta^{d+1}$. We also denote by $F^I$ the $I^{th}$ face of $\Delta^{d+1}$. With
  this notation $F^i$ contains $F^I$ if and only if $i\in I$.

  We say that $F^i$ is positive if $\sigma|_{F^i}$ is orientation
  preserving, and negative otherwise.
  We split $\{0,\dots,d+1\}$ in two sets of indices
    $$\Lambda_+\coloneqq\{i\mid F^i\text{ is positive}\} \qquad
        \Lambda_-\coloneqq\{i\mid F^i\text{ is  negative}\}
        $$
        so that we have
    \begin{equation}\label{eq:signfacet}
        (-1)^{i}\sgn(\sigma^i)=\begin{cases}
            +1\quad&\text{if }i\in \Lambda_+,\\
            -1\quad&\text{if }i\in \Lambda_-.
        \end{cases}
    \end{equation}
For any $I\subseteq\{0,\dots,d+1\}$ we define
    $I_+=I\cap \Lambda_+$ and  $I_-=I\cap \Lambda_-$. 

    \medskip

Since $\sigma$ is $d$-immersed, we can apply \cref{lemma:linearalgebra}  to
$d\sigma_{e_v}|_{T\Delta^{d+1}}:T_{e_v}\Delta^{d+1}\to T_{\sigma(e_v)}M$, and orientations can be
chosen so that $\lambda_i>0$ in \cref{lemma:linearalgebra} if and only if $i\in\Lambda_+$. 
We get the following dichotomy.

\begin{claim}\label{claim:3.18}
      For any $I\subsetneq\{0,\dots,d+1\}$ there are two cases:
      \begin{enumerate}
      \item $I=I_+$ or $I=I_+$. In this case for any $v\notin I$ we have $\sum_{i\in
          I}\theta^I_v(\sigma^i)=1.$ In particular, for $v=\max I^c$ we get
        $$\sum_{i\in I}\theta^I(\sigma^i)=1.$$
        
      \item Both $I_\pm$ are non-empty. In this case for any $v\notin I$ we have
       $\sum\limits_{i\in  I_+}\theta^I_v(\sigma^i)-\sum\limits_{i\in
         I_-}\theta^I_v(\sigma^i)=0.$
       In particular, for $v=\max I^c$ we get
       $$\sum\limits_{i\in  I_+}\theta^I(\sigma^i)-\sum\limits_{i\in I_-}\theta^I(\sigma^i)=0.$$
       
      \end{enumerate}
      \begin{proof}
        $(1)$ If $I=I_+$, then all facets around $F^I$ are positive, and by
        \cref{lemma:linearalgebra} $\sigma$ is a
      local homeomorphism of $\partial\Delta^{d+1}$ around points in the interior of $F^I$.
      This implies that $\cup_{i\in I}\Sect_v^I(\sigma^i)=T_{\sigma(e_v)}M$
      and every point is covered once. Therefore 
      $\sum_{i\in I}\theta^I_v(\sigma^i)=1.$

      $(2)$ If both $I_\pm$ are non empty, by \cref{lemma:linearalgebra}
      $\sigma|_{\partial \Delta^{d+1}}$
       has local  degree zero near $F^I$,
             $\cup_{i\in I_+}\Sect_v^I(\sigma^i)=
      \cup_{i\in I_-}\Sect_v^I(\sigma^i)$
      and every point is covered exactly once by each side. Therefore
      $\sum_{i\in I_+}\theta^I_v(\sigma^i)=\sum_{i\in I_-}\theta^I_v(\sigma^i).$
      \end{proof}
\end{claim}

 We are now ready to conclude. We start with

    \begin{equation}\label{eqn:  cocycle relation for non-degenerate}    
    \bTheta(\partial\sigma)=\sum\limits_{i=0}^{d+1}(-1)^i\bTheta(\sigma^i).
    \end{equation}
By definition of $\boldsymbol{\Theta}$ (\cref{def:angledefect}) 
    if we combine \cref{eqn:  cocycle relation for non-degenerate} and \cref{eq:signfacet}, we obtain
    \begin{align*}
      \bTheta(\partial \sigma)&=\sum_{i\in\Lambda_+}\sum_{\ell=0}^{d}(-1)^{d-\ell}
      \!\!\!\!\!\!\!\sum_{\substack{I\subset\{0,\dots,d\}\\ |I|=\ell\\}}\theta^I(\sigma^i)-
\sum_{i\in \Lambda_-}\sum_{\ell=0}^{d}(-1)^{d-\ell}
  \!\!\!\!\!\!\!\sum_{\substack{I\subset\{0,\dots,d\}\\ |I|=\ell\\}}\theta^I(\sigma^i)\\
      &=\sum_{i\in\Lambda_+}\sum_{\ell=0}^{d}(-1)^{d-\ell}
        \!\!\!\!\!\!\!\sum_{\substack{J\subset\{0,\dots,d+1\}\\ |J|=\ell+1\\i\in J}}
      \!\!\!\!\theta^J(\sigma^i)-
\sum_{i\in \Lambda_-}\sum_{\ell=0}^{d}(-1)^{d-\ell}
      \!\!\!\!\!\!\!\sum_{\substack{J\subset\{0,\dots,d+1\}\\ |J|=\ell+1\\i\in I}}
      \!\!\!\!\theta^J(\sigma^i)\\
    \end{align*}
    and by rearranging the terms in this sum, we get 
        \begin{equation}\label{eqn: cocycle 2}
          \bTheta(\partial \sigma)=\sum_{\ell=0}^{d}(-1)^{d-\ell}
          \sum_{\substack{J\subset\{0,\dots,d+1\}\\ |J|=\ell+1}}\left(
    \sum_{i\in J_+}\theta^J(\sigma^i)-\sum_{i\in J_-}\theta^J(\sigma^i)\right).
    \end{equation}
    From \cref{claim:3.18}
    \begin{equation}\label{eqn: sum opposite angles}
        \sum_{i\in J_+}\theta^J(\sigma^i)-\sum_{i\in J_-}\theta^J(\sigma^i)=
        \begin{cases}
            1 &\text{if}\ J=J_+,\\
            -1 &\text{if}\ J= J_-,\\
            0 &\text{otherwise}.
        \end{cases}
    \end{equation}

    Since $\Delta^{d+1}$ is contractible, the degree of $\sigma:\partial\Delta^{d+1}\to M$ is
    zero, hence both $\Lambda_\pm$ are not empty. Thus 
    $|\Lambda_+|\ge1$ and $|\Lambda_-|\ge1$, and from
    \cref{eqn: cocycle 2} and \cref{eqn: sum opposite angles} we get that
    \begin{align*}
      \bTheta(\partial \sigma)
      &=\sum_{\ell=0}^{d}(-1)^{d-\ell}
    \left(\binom{|\Lambda_+|}{\ell+1}-\binom{|\Lambda_-|}{\ell+1}\right)\\
      &=\sum_{\ell=1}^{d+1}(-1)^{d+1-\ell}\binom{|\Lambda_+|}{\ell}
        -\sum_{\ell=1}^{d+1}(-1)^{d+1-\ell}\binom{|\Lambda_-|}{\ell}\\
      &=(-1)^{d+1}\sum_{\ell=1}^{|\Lambda_+|}(-1)^{\ell}\binom{|\Lambda_+|}{\ell}-(-1)^{d+1}
        \sum_{\ell=1}^{|\Lambda_-|}(-1)^{\ell}\binom{|\Lambda_-|}{\ell}\\
      &=(-1)^{d+1}\left[ \big((1-1)^{|\Lambda_+|} - 1\big)
        - \big((1-1)^{|\Lambda_-|} - 1\big) \right]\\
    &=0.
    \end{align*}
\end{proof}

\section{Strict cocycle relation and continuity}\label{sec:cocornot}

\cref{thm:cocyclecondition} only shows that the angle cochain satisfies the cocycle relation on immersed simplices. It is natural to wonder if this can be extended to all simplices. Well, certainly not with our extension $\bTheta(\sigma)=0$ for non-immersed simplices (\cref{def:angledefect}). For example, if one takes an embedded $(d+1)$-simplex
(so $\bTheta(\partial\sigma)=0$ by \cref{thm:cocyclecondition}), and just collapse a neighbourhood of one vertex $v$ to a point,
then now all facets of $\sigma$ except the facet opposed to $v$ are not immersed. Therefore the
sum $\sum_i(-1)^i\bTheta(\sigma^i)$ is reduced to just the contribute of the single immersed
facet which can be non-zero.

However, let us see that under some continuity assumptions, there is another natural choice for extending $\bTheta$ to non immersed simplices leading to a strict combinatorial cocycle $\bTheta \circ \mathrm{str}$. (Recall that $\str$
is compatible with faces by \cref{pointfaces} of \cref{def:Mais})

\begin{definition}\label{def:tetacont}
  We say that the map $\str$ is $\bTheta$-continuous if the map
$$\sigma\mapsto \sgn(\sigma)\bTheta(\str(\sigma))$$
  can be continuously extended to simplices  
  such that $\str(\sigma)$ is not immersed.
\end{definition}

If $\str$ is $\bTheta$-continuous, then the value of $\sgn(\sigma)\bTheta$ is defined on
non-immersed simplices by its continuous extension.

\begin{proposition}\label{prop:isacoc}
  If the map $\str$ is $\bTheta$-continuous then $\bTheta\circ\str$ is a cocycle.
\end{proposition}
\begin{proof}
  Let $\sigma$ be a $(d+1)$-simplex. We have to compute
  $$\bTheta(\partial(\str(\sigma)))=\sum_i(-1)^i\sgn(\sigma^i)\bTheta(\str(\sigma^i)).$$
\cref{pointnond} of \cref{def:Mais} implies that there exists a sequence of simplices
$\sigma_n$ so that $\sigma_n\to \sigma$ and $\str(\sigma_n)$ is $d$-immersed. By \cref{pointfaces} of \cref{def:Mais} we have
$$\str(\sigma_n^i)\to\str(\sigma^i)$$ for all $i$.
If $\str(\sigma^i)$ is immersed then $\sgn(\str(\sigma^i_n))\to\sgn(\str(\sigma^i))$ because of
\cref{pointlem:postriangopen} of \cref{def:Mais}. On the other and if $\str(\sigma^i)$ is not
immersed then $\bTheta(\str(\sigma_n^i))$ tautologically converges to its limit because $\str$
is $\bTheta$-continuous. In any case, for all $i$ we have
$$\sgn(\str(\sigma^i_n))\bTheta(\str(\sigma^i_n))\to\sgn(\str(\sigma^i))\bTheta(\str(\sigma^i)).$$
By \cref{thm:cocyclecondition} we know $\bTheta(\partial \str(\sigma_n))=0$. So
\begin{align*}
  0&=\bTheta(\partial \str(\sigma_n))\\
   &=\sum_i(-1)^i\sgn(\str(\sigma^i_n))\bTheta(\str(\sigma^i_n))
     \to \sum_i(-1)^i\sgn(\str(\sigma^i))\bTheta(\str(\sigma^i))\\
   &=\bTheta(\partial \str(\sigma)).  
\end{align*}
\end{proof}

The next section is devoted to proving that, even without $\bTheta$-continuity, one can get things to work.

 \section{$\e$-angle coycles}\label{sec: epsilon-angle-cochain}
In this section we introduce a regularization procedure by locally averaging angle over $\e$-neighborhoods around the
vertices of a given simplex. This produces a family of continuous cochains that are indeed true
cocycles. 

The setting of this section is the following:
\begin{itemize}
\item $M$ is a $d$-manifold, $\Omega=\wt M$ is the universal covering of $M$ and
  $\Gamma=\pi_1(M)$ is the group of deck transformations of $M$, so that $M=\Omega/\Gamma$.
\item $\F$ is a Finlser structure on $M$ (for example a Riemannian structure). $\F$ induces on
  $M$ both a metric $d$ and a volume-measure $\mu$ in 
  the class of Lebesgue measure (see for instance~\cite[Section 3]{AlvThom} for details on
  Finlser structures). Also, $\F$ determines angles and angle cochains.
  We lift everything to $\Omega$.
\item We suppose that $M$ admits an immersed straighening and we fix a
  map $$\sing:\Comb(\Omega)\to\Sing(\Omega)$$ as in \cref{def:Mais}.
\item We identify $\Comb_k(\Omega)$ with $\Omega^{k+1}$ and $\Comb_k(M)$ with $M^{k+1}$,
  endowed with the product measure that we denote $\mu^{\otimes d+1}$.
\item For any $x\in \Omega$ and $\e>0$ we set $B_\e(x)=\{y\in\Omega: d(x,y)<\e\}$. For any
  $S=(x_0,\dots,x_k)\in \Omega^{k+1}$ we define $$B_\e(S)=\Pi_{i=0}^kB_\e(x_i).$$
We use that notation also for balls in $M$ and in $M^{k+1}$.  
\end{itemize}

\begin{definition}\label{def:eangdef}
 For any  $\e>0$ the $\boldsymbol{\e}${\em -angle defect} of a combinatorial $d$-simplex
    $S=(v_0,\dots,v_d)\in\Omega^{d+1}$ is
     $$
    \Theta_\e(S)\coloneqq\dfrac{1}{\mu^{\otimes d+1}(B_\e(S))}\int_{B_\e(S)}\Theta(\sing(x_0,\dots, x_d))\ d\mu(x_0)\dots d\mu(x_d).
    $$
\end{definition}
\begin{definition}\label{def:eangcoc}
    For any $\e>0$. The {\em $\e$-angle cocycle} on $\Omega$ is the linear map
    $\bTheta_\e:C_d(\Omega;\mathbb R)\to\mathbb{R}$ defined on the basis of singular
     $d$-simplices $\sigma:\Delta^d\to\Omega$ by
    $$
    \bTheta_\e(\sigma)\coloneqq\Theta_\e(\comb(\sigma))
    =\Theta_\e(\sing(\comb(\sigma)))=\Theta_\e(\str(\sigma)).
    $$

    The {\em $\e$-angle cocycle} on $M$  is the linear map 
$\bTheta_\e:C_d(M;\R)\to \mathbb{R}$ defined on singular simplices by
    \begin{equation}
        \bTheta_\e(\sigma)\coloneqq\bTheta_\e(\widetilde{\sigma}),
    \end{equation}
    where $\widetilde{\sigma}$ is any lift of $\sigma$.
    
\end{definition}
We notice that the $\e$-angle cocycle on $M$ is well defined because of equivariance of $\sing$
(\cref{pointequivariant} of \cref{def:Mais}).

\begin{theorem}\label{teo: boundedness of the cocycle}
  Let $M$ be a $d$-manifold equipped with a Finsler structure $\F$ (for instance a Riemannian metric).
  For any $\e>0$, let $\bTheta_\e$ be the $\e$-angle cochain induced by $\F$.

  Then $\bTheta_\e$ is a bounded continuous cocycle.  Moreover, there exists $C_d\leq\frac{d+2}{4}$ such that for any singular  $d$-simplex $\sigma:\Delta^d\to\Omega$ we have
  $|\bTheta_\e(\sigma)|\le C_d.$
\end{theorem}
\begin{proof}
  We will work on $\Omega$, equivariance of $\sing$ will imply the result for $M$.
  The bound on $|\bTheta_\e(\sigma)|$ immediately follows from \cref{cor:bound} by integrating. 

  To prove the continuity of $\boldsymbol{\Theta_\e}$, it suffices to show that the map $\Theta_\e(S)$  is continuous on $S$.    
  The measure $\mu$ induced by $\F$ is absolutely continuous with respect to the Lebesgue measure (see for instance \cite[Section 3]{AlvThom}), hence we have that the map  
  $
  \Omega^{d+1}\ni S\mapsto \mu(B_\e(S))
  $
  is continuous, and
  ${\mathbbm{1}}_{B_\e(S)}\in{\operatorname{L}}^1(\mu^{\otimes d+1})$.  
  Moreover, the function 
  $
  \Theta(S)
  $ 
  is bounded on $\Omega^{d+1}$ (\cref{cor:bound}). Consequently, the map $\Theta_\e$ is continuous.
  
  To prove that $\bTheta_\e$ is a cocycle, it suffices to show that for any singular $(d+1)$-simplex $\sigma:\Delta^{d+1}\to\Omega$ we have
  $\bTheta_\e(\partial\sigma)=0$.
      Fix $\sigma:\Delta^{d+1}\to\Omega$ and    let
     $S=(v_0,\dots,v_{d+1})=(\sigma(e_0),\dots,\sigma(e_{d+1}))=\comb(\sigma)$. By definition

    \begin{equation}\label{eqn: epsilon-cocycle of the boundary}    
    \bTheta_\e(\partial\sigma)=\sum\limits_{i=0}^{d+1}(-1)^i \bTheta_\e(\sigma^i).
    \end{equation}
    For any $i=0,\dots,d+1$, 
    $$\bTheta_\e(\sigma^i)
    =\dfrac{1}{\prod\limits_{\substack{j=0\\j\neq i}}^{d+1}
      \mu\left(B_\e(v_j)\right)}
    \int\limits_{\prod\limits_{\substack{j=0\\j\neq i}}^{d+1}
      B_\e(v_j)}
    \Theta(\sing(x_0,\dots,\widehat{x_i},\dots, x_{d+1}))\ d\mu^{\otimes d+1}
    $$
    and, since $\Theta(\sing(x_0,\dots,\widehat{x_i},\dots, x_{d+1}))$ does not depend on $x_i$, by
    Fubini-Tonelli Theorem we have
            \begin{equation}\label{eqn: epsilon-cocycle on a facet}
              \bTheta_\e(\sigma^i)
              = \dfrac{1}{\mu^{\otimes d+2}(B_\e(S))}
              \int\limits_{B_\e(S)}
              \Theta(\sing(x_0,\dots,\widehat{x_i},\dots, x_{d+1}))\ d\mu_\Omega^{\otimes d+2}.
            \end{equation}
            By face-compatibility of the map $\sing$ (\cref{pointfaces} of \cref{def:Mais})
            we have
            $$\sing(x_0,\dots,\widehat{x_i},\dots, x_{d+1})=\sing^i(x_0,\dots, x_{d+1})$$
            whence, by definition of angle cochain, combining \cref{eqn: epsilon-cocycle of the
              boundary} with \cref{eqn: epsilon-cocycle on a facet}, and by additivity of the integral, we get\label{pippo}
    \begin{align*}
        \bTheta_\e(\partial\sigma)
      &=\dfrac{1}{\mu^{\otimes d+2}(B_\e(S))}\int\limits_{B_\e(S)}
        \sum\limits_{i=0}^{d+1}(-1)^i \Theta(\sing^i{(x_0,\dots, x_{d+1})})\ d\mu^{\otimes d+2}\\
      &=\dfrac{1}{\mu^{\otimes d+2}(B_\e(S))}\int\limits_{B_\e(S)}\bTheta(\partial(\sing((x_0,\dots, x_{d+1}))))\ d\mu^{\otimes d+2}.
    \end{align*}

    By \cref{pointnond}  of \cref{def:Mais} the set
    $\mathcal D_{d+1}\subset \Omega^{d+2}$ of simplices $X$ such that $\sing(X)$ is not
    $d$-immersed has measure zero.
        On the complement of $\mathcal D_{d+1}$,
    \cref{thm:cocyclecondition}
    guarantees that $\bTheta\circ \partial =0$, therefore
    \begin{equation*}
        \bTheta_\e(\partial\sigma)
        =\dfrac{1}{\mu^{\otimes d+2}(B_\e(S))}\int\limits_{B_\e(S)\setminus \mathcal D_{d+1}}
        \bTheta(\partial(\sing(x_0,\dots, x_{d+1})))\ d\mu^{\otimes d+2}
        =0.
    \end{equation*}
\end{proof}

\section{Computing the Euler characteristic}\label{sect: computing chi}\label{proof of main thm}
The aim of this section is to prove that for any sufficiently small $\e>0$, the cocycle $\bTheta_\e$ computes the Euler characteristic of compact manifolds.
Roughly, the idea is to take a triangulation $\tau$ and compute $\bTheta_\e$ on its straightening.

The first issue that we have to deal with, is that straightening of triangulations could exhibit  overlapping of simplices, so we have to prove that we can find triangulations whose straightening are triangulations.

\medskip

Let $M$ be an oriented, compact smooth $d$-manifold equipped with a Finsler structure and
admitting an 
immersed straighening. We keep the notation as in
\cref{sec: epsilon-angle-cochain}. In particular we set $\Omega=\wt M$ and $\Gamma=\pi_1(M)$.

We use the following notation for triangulations:
A finite simplicial complex $K$ is a subset of $\R^{N+1}$, whose cells are convex hulls of subsets of the standard basis $e_0,\dots,e_N$ with the ordering induced by $0<\dots<N$. So $K$ is a sub-complex of $\Delta^N$. A triangulation of $M$ will be a pair $(K,\tau)$ where $K$ is a finite simplicial complex, and $\tau:K\to M$ is a homeomorphism.

We denote by $K^{(n)}$ the  $n$-skeleton of $K$.
Any $n$-cell of $K$ is the convex hull in $\mathbb R^N$ of $n+1$ ordered vertices $(e_{i_0},\dots,e_{i_n})$ with $i_0<\dots <i_n$.
For any $n=0,\dots,d$, given an $n$-cell $e$ of $K$, we denote by 
$$
\sigma_e:\Delta^n\to K
$$ the affine map induced by $(e_0,\dots,e_n)\mapsto (e_{i_0},\dots,e_{i_n})$.
A  triangulation is oriented if an orientation on (any $d$-simplex of) $K$ is chosen so that $\tau$ is orientation preserving, thus  the degree of $\tau$ is $1$. 
\begin{remark}
It is important to note that $\sigma_e$ preserves the order or vertices. So if $s$ is a cell of $K$, and $v$ is the vertex of $s$ with maximal order in $K$, then for any cell $e$ containing $s$, $\sigma_e^{-1}(v)$ is the vertex of $\sigma_e^{-1}(s)$ with maximal order. (This is crucial because of our definition of angle $\theta^I(\eta)=\theta^I_{\max I^c}(\eta)$.) 
\end{remark}

We denote by
$$\tau_e=\tau\circ\sigma_e$$
the singular simplex $\tau_e=\Delta^n\to M$ corresponding to the cell $e$. We assume that $\tau_e$ has always smooth image.
When the triangulation is oriented, for any $d$-cell $e$ we define  
$$\pm_e\coloneqq
\begin{cases}
  +1\quad & \text{if $\sigma_e$ is orientation-preserving}\\
  -1 \quad & \text{otherwise}
\end{cases}$$

Observe that $$\pm_e\sgn(\tau_e)=1.$$
We denote by $c_\tau\in C_d(M;\R)$ the singular $d$-chain 
    \begin{equation*}
        c_\tau\coloneqq\sum_{e\in K^{(d)}} \pm_e\tau_e.
    \end{equation*}

\begin{remark}\label{rem: ct}
It is well known that $[c_\tau]=[M]_\Z$ (see \cite{Hatcher}).
\end{remark}

By face-compatibility (\cref{pointfaces} of \cref{def:Mais}), the straightening map induces a continuous map $\str(\tau):K\to M$ so that 
$$\str(\tau)\circ\sigma_e=\str(\tau_e).$$
(Note that $(K,\str(\tau))$ may not be a triangulation).
We set
$$\str(\tau)_e=\str(\tau_e)\quad\text{and}\quad
c_{\str(\tau)}=\sum\pm_e\str(\tau)_e.$$
Since $\str$ is homotopic to the identity (\cref{pointeqhom} of \cref{def:Mais}), then $\str(\tau)$ is homotopic to $\tau$, and it is readily checked that
$c_{\str(\tau)}=\str(c_\tau).$ 

We address now the question of whether or not $\str(\tau)$ is a homeomorphism (hence gives a triangulation).

\begin{definition}\label{def: straightenable triangulation}
    Let $M$ be a compact oriented smooth $d$-manifold admitting an immersed straightening. 
    We say that a triangulation $(K,\tau)$ of $M$  {\em has a positive straightening} if for any $d$-cell $e\in K^{(d)}$ and for any lift $\wt{\tau_e}:\Delta^d\to
    \Omega$ of $\tau_e$, the $d$-simplex $\str(\wt{\tau_e})$ is immersed and oriented according to $\pm_e$, {\em i.e.}
    $\sgn(\str{\wt{\tau_e}})=\pm_e$ (equivalently, the restriction of $\str(\tau)$ to any $d$-cell of $K$ is orientation preserving).
\end{definition}

\begin{lemma}\label{lemma: str}
    Let $M$ be a compact oriented smooth $d$-manifold admitting an immersed straightening. 
    If $(K,\tau)$ is a triangulation of $M$  with positive straightening, then $\str(\tau)$ is an orientation preserving homeomorphism (and {\em vice versa}).
  \end{lemma}

  \begin{proof}
    The {\em vice versa} is tautological, as well as the orientation preserving claim.
    Let $(K,\tau)$ be a triangulation of $M$  with positive straightening, then $\str(\tau):K\to M$ is a continuous map. We prove that it is open and bijective.

    $(1)$ Since $(K,\tau)$ has a positive straightening, for any cell $e$,  $\str(\wt\tau_e)$ is an immersion. Moreover, the restriction of
    $\str$ on any $d$-cell of $K$ is orientation-preserving. 

    $(2)$ $\str(\tau)$ is locally injective on any cell. This follows immediately from $(1)$.

    $(3)$ $\str(\tau)$ is injective on $K\setminus K^{(d-1)}$. Moreover it is onto.
    The first claim follows from $(1)$ and the fact that $\str(\tau)$ is homotopic to $\tau$, hence it has degree one, which also implies that $\str(\tau)$ is onto.

    $(4)$ If $x,y\in K$ have the same image, then $\str(\tau)$ is not locally injective near both $x,y$: If $\str(\tau)$ was locally injective on an open set $U\ni x$, then by the Theorem of Invariance of Domain, the restriction of $\str(\tau)$ to $U$  would be open, hence an homeomorphism between $U$ and its image. By continuity there is an open set $V\ni y$ whose image is contained in $\str(\tau)(U)$. But that would contradict $(3)$.
    So  $\str(\tau)$ is not locally injective near $x$. The same holds near $y$.
    
    $(5)$ We prove that $\str(\tau)$ is globally injective on any $K\setminus K^{(d-n)}$, arguing by induction on $n$. The case $n=1$ is in point $(3)$. Suppose we proved that
    $\str(\tau)$ is injective on $K\setminus K^{(d-n)}$. We argue by contradiction and suppose that there is $x,y\in K\setminus K^{(d-n-1)}$ so that $\str(\tau)(x)=\str(\tau)(y)$. By $(4)$
    $\str(\tau)$ is not locally injective near $x,y$. Since $K\setminus K^{(n)}$ is open, the inductive hypothesis implies that that $x,y\in K^{(d-n)}\setminus K^{(d-n-1)}$, but that would contradict $(2)$. Therefore $\str(\tau)$ is injective also on  $K\setminus K^{(d-n-1)}$. 

    Now that we know that $\str(\tau)$ is continuous, onto, and injective, the Compact-Hausdorff Theorem concludes (alternatively one can use the invariance of domain to get openness).
  \end{proof}

  Triangulations with positive straightening always exist.  
  \begin{proposition}\label{prop: positively oriented triangulation}
    Let $M$ be a compact oriented smooth $d$-manifold admitting an immersed straightening.     
    Then there exists a positively oriented finite triangulation $(K,\tau)$ of $M$ that has a positive straightening.
\end{proposition} 
\begin{proof}
    Since $M$ is a closed oriented smooth manifold, it has a smooth positively oriented finite triangulation $(H,\eta)$(\cite[Theorem 10.6]{munkres}) with $H\subset \R^N$ for some
    $N$. Note that by definition of $H$ is a sub-complex of $\Delta^{N+1}$.    

    By \cite[Theorem 9.4]{munkres}, there is $t_0>0$ such that $\Delta^{N+1}$ has arbitrarily fine subdivision with  thickness at least $t_0$, and from that
    we can build an arbitrarily fine subdivision of $H$ with thickness at least $t_0$.

    By \cref{pointprop:posortri} of \cref{def:Mais}, a sufficiently fine such subdivision of $H$ will be a triangulation of $M$ with positive straightening.
\end{proof}

In order to prove that the cocycle $\bTheta_\e$ computes the Euler
characteristic, we first show that  the cochain $\bTheta$ does the job on triangulations with
positive straighening.  
\begin{lemma}\label{lem: positive straightening computes the euler characteristic}
  Let $M$ be an oriented closed smooth $d$-manifold admitting an immersed straighening, and equipped with a Finsler structure $\F$. Let $\bTheta$ be the angle cochain induced by $\F$.
  Let $(K,\tau)$ be a positively oriented finite triangulation of $M$.
  Then
    $$
        \bTheta(c_{\tau})=(-1)^d \chi(M).
    $$
Moreover, if $\tau$ has a positive straightening, then
$\bTheta(c_{\str(\tau)})=\bTheta(c_{\tau})=(-1)^d \chi(M).$
\end{lemma}
\begin{proof}
  The last claim clearly follows from the first because if  $\tau$ has a positive straightening, then by~\cref{lemma: str}, $\str(\tau)$ is an homeomorphism, so $(K,\str(\tau))$ is a positively oriented triangulation of $M$. 
  
  We fix the following notation. We remind that $K\subset \Delta^N$ and cells of $K$ of dimension $n$ are the convex hull of some set $\{e_{i_0},\dots,e_{i_n}\}$. For any $d$-cell
  $e$ of $K$ and any $n$-cell $s$ which is a face of $e$, if $E$ and $S$ are the subsets of  $\{0,\dots,N\}$ corresponding to $e$ and $s$ respectively (so $|E|=d+1$, $S\subseteq E$, and $|S|=n+1$), we set
  $$\theta_S(E)=\theta^I(\tau_e)$$
  where $I$ is the subset of $\{0,\dots,d\}$ corresponding to $E\setminus S$. (Note that $|S|=d+1-|I|$).

  Then we have

$$
      \bTheta(c_\tau)
      =\sum_{e\in K^{(d)}} \pm_e\bTheta(\tau_e)
      =\sum_{e\in K^{(d)}} \pm_e\sgn(\tau_e)
        \big(\sum_{\ell=0}^{d}(-1)^{d-\ell}
        \sum_{\substack{I\subset\{0,\dots,d\}\\ |I|=l}}\theta^I(\tau_e)\big).$$
Since $\pm_e\sgn(\tau_e)=1$, by
rearranging terms, we get
    \begin{align*}
      &\bTheta(c_\tau)=\sum_{e\in K^{(d)}}\sum_{\ell=0}^{d}(-1)^{d-\ell}
        \sum_{\substack{I\subset\{0,\dots,d\}\\ |I|=\ell}}\theta^I(\tau_e)
                     =\sum_{\ell=0}^{d}(-1)^{d-\ell}\sum_{e\in K^{(d)}}
                       \sum_{\substack{I\subset\{0,\dots,d\}\\ |I|=\ell}}\theta^I(\tau_e)\\
               &=\sum_{\ell=0}^{d}(-1)^{d-\ell}\sum_{\substack{E\subset\{0,\dots,N\}\\ |E|=d+1}}
      \sum_{\substack{S\subset E\\ |S|=d-\ell+1}}\theta_S(E)      
               =\sum_{\ell=0}^{d}(-1)^{d-\ell}\sum_{\substack{S\subset\{0,\dots,N\}\\ |S|=d-\ell+1}}
      \sum_{\substack{E\subset\{0,\dots,N\}\\ E\supset S\\ |E|=d+1}}\theta_S(E).\\      
    \end{align*}
    Since $\tau$ is an homeomorphism, the sum of all
    angles around any simplex $\tau_s$ is exactly $1$. That is to say
    $\sum_{\substack{E\subset\{0,\dots,N\}\\ |E|=d+1\\E\supset S}}\theta_S(E)=1$. Thus
$$
\bTheta(c_\tau)
=\sum_{\ell=0}^{d}(-1)^{d-\ell}\sum_{\substack{S\subset\{0,\dots,N\}\\ |S|=d-\ell+1}}
      1
            =\sum_{\ell=0}^{d}(-1)^{d-\ell}\sum_{s\in K^{(\ell)}}
            1
            =(-1)^d \chi(M).$$

\end{proof}

Since we want to show that also $\bTheta_\e$ computes the Euler characteristic, for a suitable choice of $\e>0$, we have to show that ``having a positive straightening'' is a stable condition. This will be the content of \cref{lem: positive smooth triangulations are open}.

First, we need to fix a precise way of choosing lifts of perturbations of combinatorial simplices. Let $V\subset M$ be a finite set. Suppose that for a certain $\e>0$ the balls
$B_\e(v)$ and $B_\e(w)$ are disjoint for any $v\neq w$ in $V$. Let $\sigma:\Delta^k\to M$  be a singular $k$-simplex in $M$ with distinct vertices, so that there are $v_0,\dots,v_k\in V$ such that $\comb(\sigma)=(v_0,\dots,v_k)$. (For example, this is the case when $\sigma$ is a simplex in a triangulation with vertex set $V$).
Let $\comb(\wt\sigma)=(\wt v_0,\dots, \wt v_k)$ for some lift $\wt\sigma$ of $\sigma$. For any $X=(x_v)\in\prod_{v\in V} B_\e(v)$ we define
\begin{equation}
  \label{eq:liftcomb}
  \wt S_\e(\wt\sigma,X)\coloneqq(\wt x_0,\dots,\wt x_k)
\end{equation}
to be  the unique combinatorial simplex so that any $\wt x_i$ is the unique lift of $x_i$ belonging
to $B_\e(\wt v_i)$. Observe that because  of the equivariance of the straighening
(\cref{pointequivariant} of \cref{def:Mais}) the angles of $\str(\wt S_\e(\wt\sigma,X))$ do not depend on the chosen lift $\wt\sigma$. Moreover, it
is important to note that 
$\wt S_\e(\wt\sigma,X)$ does not depend on variables $x_v$ with $v\notin\{v_0,\dots,v_k\}$.
Such notation is useful to compute
$\bTheta_\e$.

\begin{lemma}\label{lemma: lift}
Let $M$ be a closed oriented Finsler $d$-manifold, admitting an immersed straightening.

Let $V\subset M$ be a finite set of cardinality $n$. Let $\e>0$ be so that the balls $B_\e(v)$'s are pairwise disjoint as $v$ varies in $V$.  Let $\sigma$  be a singular
$d$-simplex in $M$ with distinct vertices in $V$. Then
$$\bTheta_\e(\sigma)=
\frac{1}{\mu^{\otimes n}(\prod_{v\in V}B_\e(v))}
\int_{\prod_{v\in V}B_\e(v)}
\Theta(\sing(\wt S_\e(\wt\sigma, X)))\ d\mu^{\otimes n}(X).$$
\end{lemma}
\begin{proof}
  \begin{align*}
  &\boldsymbol{\Theta_\e}(\sigma)
  = \boldsymbol{\Theta_\e}(\wt
    \sigma)=\Theta_{\e}(\comb{\wt\sigma})\\
  &=\frac{1}{\prod_{i=0}^d\mu(B_\e(\wt\sigma(e_i)))}
    \int_{\prod_{i=0}^dB_\e(\wt\sigma(e_i))}
    \Theta(\sing(\wt x_0,\dots,\wt x_d))\ d\mu(\wt x_0)\dots
    d\mu(\wt x_d)\\
  &=\frac{1}{\prod_{i=0}^d\mu(B_\e(\sigma(e_i)))}
    \int_{\prod_{i=0}^dB_\e(\sigma(e_i))}
    \Theta(\sing(\wt S_\e(\wt\sigma, X)))\ d\mu(x_0)\dots
    d\mu(x_d)\\
  &=\frac{1}{\mu^{\otimes n}(\prod_{v\in V}B_\e(v))}
    \int_{\prod_{v\in V}B_\e(v)}
    \Theta(\sing(\wt S_\e(\wt\sigma, X)))\ d\mu^{\otimes n}(X)  
\end{align*}
where the last equality follows from the Fubini-Tonelli theorem.
\end{proof}

\begin{lemma}\label{lem: positive smooth triangulations are open}
  Let $M$ be an oriented closed smooth $d$-manifold, equipped with a Finsler structure, and
  suppose that $M$ an immersed straighening. 

  Let $(K,\tau)$ be a finite triangulation and denote by $V=\tau(K^{(0)})$ the set of its vertices. If $(K,\tau)$ has a positive straightening, then there exists $\e>0$ such that
    $$B_\e(v)\cap B_\e(w)=\emptyset \quad\text{ for all }\quad v\neq w\in V,$$ and such that for all $X=(x_v)\in \prod_{v\in V}B_\e(v)$, there exists a positively oriented triangulation $(K,\tau^X)$ of $M$ such that:
    \begin{enumerate}
        \item $\tau^X(K^{(0)})=X$.
        \item $\tau^X$ has a positive straightening.
        \item For any $k$-cell $e$ of $K$ and any lift $\wt{\tau_e}$ of $\tau_e$, there
          is unique  lift $\wt{\tau_e^X}$ of $\tau^X_e$ such that $\wt S_\e(\wt{\tau_e},X)=\comb(\wt{\tau_e^X})$.
    \end{enumerate}
\end{lemma}
\begin{proof}
  Since $M$ is compact its injectivity radius is bounded below. 
  Then, for any $\e>0$ small enough, for any $p\in M$ the set $\overline{B_\e(p)}$ is a disc. Moreover, we can choose $\e$ small  enough so that, in addition, the $B_\e(v)$'s are pairwise disjoint.
  For any $v\in V$ and for any $x_v\in B_\e(v)$, the Homogeneity Lemma  \cite{milnor1965topology} implies that there  exists an ambient isotopy $h_v$ of $M$ supported in $B_\e(v)$ that maps $v$ to $x_v$. The composition of  this isotopy with $\tau$ gives a positively oriented smooth triangulation of $M$.
  Since the balls $B_\e(v)$'s are pairwise disjoint, we can compose all isotopies $h_v$'s  and obtain a positively oriented smooth triangulation $(K,\tau^X)$ of $M$ with vertices $\{x_v \ : \ v\in V\}$.
  Property $(3)$ is ensured by the fact that the balls $B_\e(v)$ are pairwise disjoint, and so their lifts are pairwise disjoint too.

  Moreover, since $K$ is finite, if $\e$ is sufficiently small, $(K,\tau^X)$ inherits from $(K,\tau)$ the fact that it has a positive straightening by \cref{pointlem:postriangopen}
  of \cref{def:Mais}. Hence, also property $(2)$ holds.
\end{proof}

Applying the results of~\cref{prop: positively oriented triangulation},
\cref{lem: positive straightening computes the euler characteristic},
and~\cref{lem: positive smooth triangulations are open},
we get the main result of this section. 
\begin{theorem}\label{thm: the angle cocycle computes the euler characteristic}
  Let $M$ be a closed oriented $d$-manifold, admitting an immersed straighening. Let $\F$ be a Finsler structure on $M$ ant let $\bTheta$ and $\bTheta_\e$ be the angle cocycles induced by $\F$.
  
  Then, there exists $\e>0$ such that 
    $$
        \chi(M)=(-1)^d [\bTheta_\e]\left([M]\right).
    $$
\end{theorem}
\begin{proof}
    Since $\bTheta_\e$ is a cocycle, it suffices to show that there exists $\e>0$ and $c\in C_d(M;\mathbb{R})$ such that $[c]=[M]\in H_d(M;\mathbb{R})$ and $\chi(M)=\bTheta_\e(c)$. 
    
    By \cref{prop: positively oriented triangulation} there exists a positively oriented finite triangulation $(K,\tau)$ that has a positive straightening. The singular chain
    $$
    c_\tau=\sum_{e\in K^{(d)}}\pm_e\tau_e
    $$
    represents the fundamental class of $M$.

    For any $e\in K^{(d)}$, we choose a lift $\wt{\tau_e}$ of $\tau_e$, and consider the singular chain    
    $$
    \wt{c_\tau}\in C_d(\Omega;\mathbb{R})\qquad
        \wt{c_\tau}\coloneqq\sum_{e\in K^{(d)}}\pm_e\wt{\tau_e}.
    $$
    By definition, 
    \begin{equation*}
        \bTheta_\e(c_\tau)=\bTheta_\e(\wt{c_\tau})=\sum_{e\in K^{(d)}}\pm_e\Theta_\e(\comb(\wt{\tau_e})).
      \end{equation*}

    Let $V=\tau(K^{(0)})=\{v_1,\dots,v_n\}$ be the set of vertices of $\tau$.
    By \cref{lem: positive smooth triangulations are open}, there exists $\e>0$ such that for any $X=(x_1,\dots,x_n)$, with $x_i\in B_\e(v_i)$ for all
    $i=1,\dots,n$, there exists a finite triangulation $(K,\tau^X)$ of $M$ that has positive straighening and vertex set $X$.
    Moreover, from point $(3)$ of \cref{lem: positive smooth triangulations are open}, for any $d$-cell $e$ of $K$, there is a unique lift $\wt{\tau_e^X}$ so that $\wt S_\e(\wt{\tau_e},X)=\comb(\wt{\tau_e^X})$.
    So
    $\Theta(\sing(\wt S_\e(\wt{\tau_e},X)))
    =\Theta(\str(\wt{\tau^X_e}))
    =\Theta(\str(\tau_e^X))=\bTheta(\str(\tau_e^X))$.
    By \cref{lemma: lift}
    $$\bTheta_\e(\tau_e)
    =\frac{1}{\mu^{\otimes n}(\prod_{v\in V}(B_\e(v)))}
    \int_{\prod_{v\in V}B_\e(v)}
    \bTheta(\str(\tau^X_e))\ d\mu^{\otimes n}(X).$$

    From that, and the fact that $\bTheta({c_{\str(\tau^X)}})=\chi(M)$
    (\cref{lem: positive straightening computes the euler characteristic}), we get
    \begin{align*}
        \bTheta_\e(c_\tau)
        &=\dfrac{1}{\mu^{\otimes n}(\prod\limits_{v\in V}(B_\e(v)))}
          \int\limits_{\prod\limits_{v\in V} B_\e(v)}
          \bTheta(c_{\str(\tau^X)}) \ d\mu^{\otimes n}(X)\\
        &=\dfrac{1}{\mu^{\otimes n}(\prod\limits_{v\in V}(B_\e(v)))}
        \int\limits_{\prod\limits_{v\in V} B_\e(v)}
        (-1)^d \chi(M) \ d\mu_M^{\otimes n}(X)\\
        &=(-1)^d \chi(M).
    \end{align*}

\end{proof}

As a corollary of \cref{thm: the angle cocycle computes the euler characteristic}, we can prove our main \cref{cor: answer to gromov question}: 

\Inequality*

\begin{proof}
    Let $T=\sum_ia_i\sigma_i$, with $a_i\in\mathbb{R}$ and $\sigma_i:\Delta^d\to M$, be a representative of the real fundamental class $[M]_\mathbb{R}$ of $M$. By \cref{thm: the angle cocycle computes the euler characteristic} there exists $\e>0$ such that
    \begin{equation}
        \chi(M)=(-1)^d \bTheta_\e(T).
    \end{equation}
    As a consequence, and using the constant $C_d\leq\frac{d+2}{4}$ provided by
    \cref{teo: boundedness of the cocycle},  
    $$
        |\chi(M)|=|\bTheta_\e(T)|=|\sum_ia_i\bTheta_\e(\sigma_i)|
            \le\sum_i|a_i||\bTheta_\e(\sigma_i)|
            \leq C_d\sum_i|a_i|=C_d||T||_1
    $$
    whence
    $$|\chi(M)|\leq C_d\inf_{[T]=[M]}||T||_1=C_d||M||.$$
    \end{proof}

We notice that \cref{pointprop:posortri} of \cref{def:Mais} is used to ensure the existence and stability of triangulations with positive straightening. And stability is needed to prove result about the $\e$-cocycles.
    
So in cases where the straighening is known to be $\bTheta$-continuous, in order to prove \cref{cor: answer to gromov question}, one does not need any other effort, a part proving that $M$ has a triangulation with a positive straightening.

We can therefore have a statement where condition $(\ref{pointprop:posortri})$ of \cref{def:Mais} is not needed.

\begin{theorem}\label{thm:no8}
  Let $M$ be an oriented closed smooth $d$-manifold,  admitting an immersed straighening, not necessarily satisfying condition $(\ref{pointprop:posortri})$ of
\cref{def:Mais}. Let $\F$ be a Finsler structure on $M$ and let $\bTheta$ be the angle cochain induced by $\F$. Suppose that
  $\str$ is $\bTheta$-continuous, and suppose that there is a triangulation of $M$ with positive straightening. Then $$|\chi(M)|\leq \frac{d+2}{4}||M||.$$
\end{theorem}
\begin{proof}
  By \cref{cor:bound} and \cref{prop:isacoc}, $\bTheta\circ\str$ is a bounded cocycle; since $M$ has a triangulation with positive straightening, by \cref{lem: positive straightening computes the euler characteristic}, it computes $\chi(M)$. The thesis follows.
 \end{proof}

\section{Convex projective manifolds}\label{sec:cpm}
In this section we obtain the comparison theorem for convex projective manifolds.
\Comparison*
\begin{proof}
In view of \cref{cor: answer to gromov question} it suffices to show that convex projective manifolds admit   immersed straightenings in the sense of \cref{def:Mais}:     
 
\begin{theorem}\label{thm cvx imm}
  Let $\Omega\subset \R\Pro^d$ be a properly convex domain. Let $\Gamma<PGL(\Omega)$ be a discrete subgroup of orientation preserving isometries of $\Omega$ and let
  $M=\Omega/\Gamma$. Then $M$ admits an immersed straightening.
\end{theorem}

The immersed straightening will be natural: The straightening of a simplex will have as image the affine hull of its vertices in an affine chart. However, some care has to be given to the parametrisation in order to obtain an equivariant straightening for the action of $PGL(\Omega)$. The parametrisation will be obtained by a classical recursive coning construction.

We first recall some background material on convex projective geometry.
A detailed treatment can be found in \cite{BenoistSurvey},\cite{benzecri}, and~\cite{GraziaThesis}.

\begin{definition}
Let $\Omega\subset\R\Pro^d$ be a non-empty open subset of the real projective space of dimension $d$. We say that $\Omega$ is a {\em properly convex domain} if it projects to a bounded convex domain in some affine chart.
\end{definition}

Any properly convex domain $\Omega\subset\R\Pro^d$ is naturally endowed with a distance function called the Hilbert metric, defined as follows. 
Let $x$ and $y$ be two distinct points in $\Omega$. Since $\Omega$ is bounded and convex in some affine chart, the projective line passing through $x$ and $y$ intersects $\partial\Omega$ in exactly two points. We denote these boundary points by $x_\infty$ and $y_\infty$, labeled such that the four points are ordered along the segment in $\overline{\Omega}$ joining them as $x_\infty, x, y, y_\infty$ (see Figure \ref{fig: HilbertDistance}).
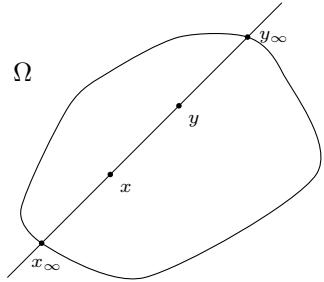
\begin{figure}[h]
    \centering
    \begin{tikzpicture}[x=.6ex,y=.6ex]
      \coordinate (x)  at (0,0);
      \coordinate (y) at (10,10);
      \coordinate (a) at (-10,-10);
      \coordinate (b) at (20,20);

      \foreach \q in {a,b,x,y}
      {\fill (\q) circle (.4);
      
      }
      \node[yshift=-8, xshift=2] at (a) {\scriptsize$x_\infty$};
      \node[xshift=10] at (b) {\scriptsize$y_\infty$};
      \node[below right] at (x) {\scriptsize$x$};
      \node[below right] at (y) {\scriptsize$y$};
      \node[left] at (-10,15) {$\Omega$};
      \draw plot[smooth] coordinates {(a) (-13,-5) (-6,10)(0,15) (10,20) (b) (25,15) (30,0)
        (5,-15) (a)};
      \draw[very thin] (-15,-15) --(25,25);
      
    \end{tikzpicture}
    
    \caption{Hilbert metric.}
    \label{fig: HilbertDistance}
\end{figure}

The {\em Hilbert metric} between $x$ and $y$ is defined as
$$
    d_\Omega(x,y)\coloneqq\dfrac{1}{2}\log[x_\infty,x,y,y_\infty],
$$
where $[x_\infty,x,y,y_\infty]$ denotes the cross-ratio of these four collinear points. By definition, we set $d_\Omega(x,x)=0$ for any $x\in\Omega$.
The Hilbert metric can be computed in any affine chart $\mathbb{A}$ where $\Omega$ is bounded as
$$
d_\Omega(x,y)=\dfrac{1}{2}\log\left(\dfrac{||x-y_\infty||}{||x-x_\infty||}\dfrac{||y-x_\infty||}{||y-y_\infty||}\right),
$$
where $||\cdot||$ is an Euclidean norm on $\mathbb{A}$. It follows that Euclidan distance on $\mathbb A$ and $d_\Omega$ induces on $\Omega$ the same topology. The group $PGL(\Omega)$ of projective transformations that preserve $\Omega$ is a subgroup of the isometries of $\Omega$ because they preserve the cross-ratio. Except for few exceptions $PGL(\Omega)$ is the full group of isometries of $\Omega$.

The Hilbert metric is induced by a Riemannian metric if and only if the domain is an ellipsoid \cite{KellyStraus}, and hence a projective model for the hyperbolic space
\cite[(29.3)]{BusemannKelly}. 
In general, the Hilbert metric is induced by a Finsler structure, that can be described as follows.
We identify any tangent space $T_p\Omega$  with $\R^d$. The {\em Hilbert norm} on the tangent space of $\Omega$ at $p$ is
    $$
    {||\xi||}_\Omega\coloneqq \frac{\operatorname{d}}{\operatorname{d}t}\Big\rvert_{t=0}d_\Omega(p,p+t\xi),
    $$
where $p\in\Omega$ and  $\xi\in T_p\Omega\cong\R^d$ is a tangent vector.
    It can be checked that in the affine chart, with labels as in \cref{fig: HilbertDistance},
    if $\xi$ is a vector pointing from $x$ to $y$ then
    $$||\xi||_\Omega=\frac{1}{2}||\xi||\left(\frac{1}{||x-y_\infty||}+\frac{1}{||x-x_\infty||}\right).$$
We notice that $\Omega$ can also be endowed with a $PGL(\Omega)$ invariant Riemannian metric, called the Blaschke metric.  We refer to \cite{BBI}, \cite{AlvThom}, \cite{ColbVern},\cite{GraziaThesis} for further details.

Any projective segment in a projective convex domain is a geodesic with respect to the Hilbert metric (but not for the Blaschke metric). These are typically
called {\em straight geodesics}. However, in general the Hilbert metric is not uniquely geodesic \cite[Proposition 2]{DeLaHarpe}. We will consider straight projective segments as our preferred geodesics.

\begin{definition}
  We say that $M$ is a real {\em convex projective manifold} of dimension $d$ if there exists a properly convex domain $\Omega\subset\R\Pro^d$ and a discrete group $\Gamma<PGL(\Omega)$ such that $M\simeq\Omega/\Gamma$. (This is equivalent to say that $M$ admits a $(\R\Pro^d,PGL_{d+1}\R)$-structure whose developing map is a homeomorphism onto a properly convex domain $\Omega\subset\R\Pro^d$).
\end{definition}

Any convex projective manifold inherits the Hilbert metric and the Finlser structure from
$\Omega$.

Convex projective manifolds have a natural straighening given by recursive coning form
vertices. More precisely, we define $\sing:\Comb_k(\Omega)\to \Sing_k(\Omega)$ inductively on
$k$ as follows: For $k=0$  we just set $\sing(x_0)(e_0)=x_0$. Suppose we have defined $\sing:\Comb_{k-1}(\Omega)\to \Sing_{k-1}(\Omega)$ and  let $S=(x_0,\dots,x_k)$ be a combinatorial $k$-simplex in $\Omega$. Since $(x_1,\dots,x_k)$ is a
      $(k-1)$-combinatorial simplex,  the singular simplex $\sigma=\sing(x_1,\dots,x_k)$ is defined. For any $p\in \Omega$, let $\gamma_p(t)$ be the straight segment from $x_0$ to $p$
    parametrized as a constant speed geodesic. Finally define
    $$\begin{array}{rccl}
    \sing(S):&\Delta^k&\longrightarrow&\Omega\\
    &(t_0,\dots,t_k)&\longmapsto& \gamma_{\sigma(\frac{t_1}{1-t_0},\dots,\frac{t_k}{1-t_0})}(t_0).
    \end{array}$$

Observe that in any affine chart for $\Omega$, the image of $\sing(S)$ is the convex hull of the vertices of $S$. We notice that the recursive coning construction is classical for
non-positively curved manifold (see for instance~\cite[Chapter 8]{frigerio2017bounded}) but the only role of curvature is the existence of unique geodesics between points. In our
setting, we have existence of unique straight geodesics, and everything works {\em mutatis mutandis} (except some smoothness results that may fail in the Finsler setting).

\begin{proof}[Proof of \cref{thm cvx imm}] We only need to check that the map $\sing$ defined by recursive coning as above satisfies all conditions of
  \cref{def:Mais}.
  \cref{pointcombsing}  immediately follows from the definition of $\sing$. \cref{pointsmooth} follows because $\sing$ is graded by construction and since the image of
  $\sing(x_0,\dots,x_k)$ is the convex hull of $(x_0,\dots,x_k)$, then $\sing(x_0,\dots,x_k)$ has smooth image. (We observe that in a Finsler setting the exponential map could fail to be smooth).
  Also \cref{pointfaces} comes for free by the recursive construction, although note that the order of vertices is crucial here. 
  \cref{pointequivariant} follows because the constant speed parametrisation of straight  segments is invariant by projective isometries of $\Omega$. \cref{pointnond} follows because $\sing(x_0,\dots,x_k)$ is not $d$-immersed if and only if there is a subset of $d+1$ points that are projectively dependent. Such a set of tuples has measure zero in $\Omega^k$.
  \cref{pointeqhom} follows because any two maps with values in $\Omega$ are homotopic via straight geodesics join, and this homotopy is equivariant with respect to the action of $PGL(\Omega)$.
  \cref{pointlem:postriangopen} is clear because the convex hull of $k$ points is $C^1$ continuous.
  
  It remains to check \cref{pointprop:posortri}. In the projective case this is easy because
  in any affine chart for $\Omega$, the convex hull of $x_0,\dots,x_d\in \Omega$ is positively
  oriented if and only if the  vectors $v_i=x_i-x_0$ form a positive basis of $\R^d$. So let us
  fix an affine chart for $\Omega$. 
  Let $\eta:\Delta^d\to \Omega$ be an immersed smooth simplex, say positive, so its Jacobian
  determinant  $\Jac_x$ is positive at any point. Since $\Delta^d$ is compact, there is $C>0$
  so that $\Jac_x>C$ for all   $x\in\Delta^d$.

  Since $\eta$ is smooth and $\Delta^d$ is compact, then $\eta$ is uniformly locally
  approximated by its   differential. 
  More precisely, for any $\delta>0$ there is $\e>0$ so that if $x,y\in\Delta^d$ are
  $\e$-close, then $$||\eta(x)-\eta(y)-d\eta_x(x-y)||<\delta.$$
  It follows that for any $\delta_1$ there is $\e>0$ so that for all $\e$-close points
  $x_0,\dots,x_d$ if 
  we set $A$ the matrix with columns $x_i-x_0$ 
  and $L$ the matrix having columns $\eta(x_i)-\eta(x_0)$, we have
  $$||L-d\eta_{x_0}A||<\delta_1.$$
    (See for instance \cite[Chapter 9]{munkres}.)
  Therefore for a sufficiently fine subdivision of $\Delta^d$, with uniform thickness (which provides a bound on $\det(A)$), if  $x_0,\dots,x_d$ are the vertices of $d$-cell of such subdivision, then $|\det(L)-\Jac_x(\eta)|<C/2$. Which implies $\det(L)>0$.
  That is to say, the convex hull of $\eta(x_0),\dots,\eta(x_d)$ is positively oriented. So $\sgn(\sing(\eta(x_0),\dots,\eta(x_d)))=\sgn(\eta)$.
\end{proof}

The proof of \cref{thm: comparison theorem} is now complete.
    \end{proof}
    
\section{Comments on definitions}\label{sec:comments_cocycle}
\subsection{Further possible angle defect cocycles}

Beside the angle defect cocycle from \cref{def:angledefect} which we worked with in this paper, its symmetrized variant (Subsection \ref{sym variant}), or more generally any cocycle obtained from a chosen average of points on faces (see \cref{rem average}) for which all statements hold with minor modifications, another possibly promising direction exploits the Schl\"affli Reduction formula: Let $\sigma$ be a non-degenerate singular simplex in $X$, where $X$ is either the sphere $S^d$, the Euclidean space $\mathbb{R}^d$, or the hyperbolic space $\mathbb{H}^d$. Then the volume of $\sigma$ satisfies 

$$    \sgn(\sigma)\sum_{F \ 2k\text{-dim face of } \sigma}c_k \alpha_F(\sigma)=\left\{ 
\begin{array}{ll}
  \frac{2}{\mathrm{Vol}(S^d)} \mathrm{Vol}(\sigma)
   & \mathrm{if \ }X=S^d\\ &  \mathrm{and \ }d\mathrm{ \ is \ even},\\
\frac{(-1)^{d/2}2}{\mathrm{Vol}(S^d)} \mathrm{Vol}(\sigma) & \mathrm{if \ }X=\mathbb{H}^d\\ & \mathrm{and \ }d\mathrm{ \ is \ even},\\
  \ \\
  0&\mathrm{otherwise},
\end{array}
\right.$$ 
for some constants $c_k\in \mathbb{Q}$ remarkably independent of the dimension $d$. The first values of these constants are $c_0=1,c_1=-1/2,c_2=1,\dots$ 

If $a_k$ denotes the number of $k$-dimensional simplices in a triangulation of a $d$-dimensional manifold $M$, then its Euler characteristic is obtained by the well known formula
$$\chi(M)=\sum_{k=0}^d (-1)^k a_k.$$
In fact exploiting the so called Dehn-Sommerville relations between the $a_k$'s, it is easy to see that $\chi(M)$ can be computed, precisely as in the Schl\"affli Reduction formula, using only the even dimensional terms:  
\begin{equation}\label{chi even}
\chi(M) =\sum_{k=0}^{d/2} c_ka_{2k},
\end{equation}
see \cite{Gav}. 
For example for surfaces one recovers the well known formula
$$\chi(M^2)=a_0-\frac{1}{2}a_2,$$
which immediately follows from the relation $3a_2=2a_1$. In higher dimensions, the additional relations stem from the vanishing of the Euler characteristic of the boundary of the tubular neighborhood of any $k$-dimensional skeleta of the triangulation. In dimension $4$, beside the obvious $5a_4=2a_3$ one also gets the relation  $2a_1-3a_2+4a_3-5a_4=0 $ which yields
$$\chi(M^4)=a_0-\frac{1}{2}a_2+a_4.$$

Our angle cocycle from \cref{def:angledefect} and its symmetrised variant were tailored so that when evalutated on a triangulation each $k$-dimensional simplex is counted with multiplicity $(-1)^k$ from contributions given by all of its adjacent top dimensional simplices. We can apply the same strategy, exploiting Formula (\ref{chi even}), so that each even dimensional $2k$-simplex is counted with multiplicity $c_k$ and odd dimensional simplices are disregarded. We hence define an alternative cochain as follows.

For every $i=0,\dots, d$ we set
\begin{equation}
  \label{eq:defalpha}
\alpha^i(\eta)=\sum_{\substack{I\subset\{0,\dots,d\}\\|I|=i}}\bar\theta^I(\eta) \qquad \text{and}\qquad\alpha_i(\eta)=\alpha^{d-i}(\eta)
  \end{equation}
so that $\alpha_i(\eta)$ is the sum of angles around faces of $\eta$ of dimension $i$ (because
$\alpha^i$ is the contribution of faces of codimension $i$). Then we define 

\begin{equation}\label{eq:volumegnrl}
  \wh\Theta(\eta)\coloneqq \sgn(\eta)\sum_{\ell=0}^{d/2}c_l\alpha_{2\ell}(\eta)
\end{equation}
(Note that  $\wh\Theta(\eta)=
      \sgn(\eta)
      (\sum_{\ell=0}^{d/2} c_\ell
        \sum_{\substack{I\subset\{0,\dots,d\}\\|I|=d-2\ell}}
        \bar\theta^I(\eta)).$)
        For example, in dimension $2$, one gets
$$    \wh\Theta(\eta)=\sgn(\eta)
      \left(\sum_{|I|=2}\bar\theta^I(\eta)-\frac{1}{2}\right),$$
while in dimension $4$ one has
    \begin{equation}\label{eq:volume}
      \wh\Theta(\eta)=\sgn(\eta)
      \left(\sum_{|I|=4}\bar\theta^I(\eta)-\frac{1}{2}\sum_{|I|=2}\bar\theta^I(\eta)
      +1\right).
  \end{equation}
Note that last terms $-1/2$ and $+1$ in these formulas are just $c_{d/2}\cdot \sum_{|I|=0}\bar\theta^I(\eta)=c_{d/2}$ since the only $d$-dimensional face of $\eta$ is $\eta$ itself. We can  promote $\wh\Theta$ to a cochain $\wh\bTheta$ by linearity, and define $\e$-cochains $\wh\bTheta_\e$. The same proofs apply to show that $\wh\bTheta$   induces a
cocycle on immersed simplices representing the Euler class.

\begin{theorem}\label{teo:Thetaeven}
If $d$ is even, then  $\wh\bTheta=\bar\bTheta$.
\end{theorem}
\begin{proof}
  We have to prove that for any $d$-simplex $\eta$ we have $\wh\bTheta(\eta)=\bar\bTheta(\eta)$.
  If $\eta$ is not $d$-immersed then $\sgn(\eta)=0$ and there is nothing to prove. So we tacitly assume to work with $d$-immersed simplices, and it is not restrictive to assume
  that $\sgn(\eta)=1$.

The proof follows the line of~\cite{Gav} plus some calculation trick as in \cref{lemma:formulacoc}. Let's recap first on the line of~\cite{Gav}.
Let $M$ be a $d$-dimensional manifold with $d$ even. Let $a_k$ denote as above the number of $k$-simplices in a given triangulation $\tau$ of $M$. The key observation is that the links of even dimensional cells of $\tau$ are spheres of odd dimension, so $\chi(Lk(e))=0$ for any cell $e$ of dimension even. Counting how many times $k$-simplices contributes to $\chi(Lk(e))$, one gets linear equations, one for each dimension.

Let's go in details. Fix $e$ a cell of $\tau$ of dimension $k-1$ (hence with $k$ vertices) so $k$ is odd. For $i=1,\dots,d-k+1$ define the $i$-valence of $e$
as $$Val_i(e)=\sharp\{S\supset e: \dim(S)-\dim(e)=i\}.$$
Any cell $S\supset F$ of dimension $\dim(e)+i$ contributes with a $(i-1)$-simplex in $Lk(e)$.
We get
$$\sum_{i=1}^{d-k+1}(-1)^iVal_i(e)=-\chi(Lk(e))=0.$$ 
Now we fix $k$ and we sum over all $(k-1)$-dimensional cells of $\tau$
$$0=\sum_{\dim(e)=k-1}\sum_{i=1}^{d-k+1}(-1)^iVal_i(e)=
\sum_{i=1}^{d-k+1}(-1)^i\sum_{\dim(e)=k-1}Val_i(e).
$$
Any
$(k-1+i)$-dimensional cell $S$ has $\binom{k+i}{k}$ faces of dimension $k-1$, because $S$ has $k+i$ vertices and a codimension $i$ face of $S$ has $k$ vertices.
It follows that $$\sum_{\dim(e)=k-1}Val_i(e)=a_{k-1+i}\binom{k+i}{k}.$$

Therefore, for any $k$ odd, numbers $a_j$'s satisfy the $k$-equation
\begin{equation}
  \label{eq:k-equation}
0=\sum_{i=1}^{d-k+1}(-1)^ia_{k-1+i}\binom{k+i}{k}  
\end{equation}

The content of~\cite{Gav} can be summarised as follows:
\begin{lemma}[\cite{Gav}] 
There are explicit numbers $c_i$ such that, whenever $(x_0,\dots,d_d)$  satisfies the $k$-equations (\ref{eq:k-equation}) for all odd $k$, then it also satisfies
$$\sum_i(-1)^ix_i=\sum_{s=0}^{d/2}c_sx_{2s}.$$
\end{lemma}
\begin{proof}
  We refer to~\cite{Gav} for a complete proof, which, in few words, follows from the fact that the  matrix of coefficients of the linear system given by the $k$-equations for $k=1,3,\dots$ is upper triangular with non zero pivots in odd positions.
\end{proof}
In particular this gives \cref{chi even}. But it also can be used to finish the proof of \cref{teo:Thetaeven}: Indeed it suffices to prove that angles $\alpha_j$'s satisfies the $k$-equations.

\begin{lemma}
  Let $\eta:\Delta^d\to M$ be an immersed smooth $d$-simplex.
  Let $\alpha_j\coloneqq\alpha_j(\eta)$ be defined as in \cref{eq:defalpha}. Then for all odd $k$ we have
  $$0=\sum_{i=1}^{d-k+1}(-1)^i\alpha_{k-1+i}\binom{k+i}{k}.$$
\end{lemma}
\begin{proof}
  Fix $I\subseteq\{0,\dots,d\}$ of even cardinality  $|I|=d-k+1$, so that $\eta^I$ is a $(k-1)$-dimensional simplex. 
  By definition (see \cref{defsect3}), for any vertex $v\notin I$ we have
  $$
  \theta_v^I(\eta)=\mu_\F(\Sect_v^I)=\mu_\F(d\eta_{e_v}(\Sect^I))=
  \mu_\F(d\eta_{e_v}(\cap_{i\in I}\Pi^i_+)).
  $$
  We observe that, since $d\eta$ is an isomorphism, and since $\partial \Pi_+^i$ and $\partial
  \Pi_-^i$ have null measure, we  have $$\mu_\F(d\eta_{e_v}((\Pi_+^i)^c)=\mu_\F(d\eta_{e_v}(\Pi_-^i)).$$ Moreover, since Finsler unit balls are symmetric with respect to the  origin, we have
  $$\mu_\F(d\eta_{e_v}(\Pi_-^i))=\mu_\F(d\eta_{e_v}(\Pi_+^i))$$ 
  and the same holds for finite intersections. To avoid heavy notations, for the next calculation we set $\mu_v(\cdot)=\mu_\F(d\eta_v(\cdot))$.
  Now, by inclusion-exclusion principle we get
  \begin{align*}
    1 & = \mu_v(\Sect^I)+\mu_v((\Sect^I)^c)=\mu_v(\Sect^I)+\mu_v((\cap_{i\in I}\Pi^i_+)^c)\\
      & = \mu_v(\Sect^I)+\mu_v(\cup_{i\in I}\Pi^i_-) = \mu_v(\Sect^I)
        +\sum_{\emptyset\neq J\subseteq I}(-1)^{|J|+1}\mu_v(\cap_{i\in J}\Pi^i_-)\\
    & = \mu_v(\Sect^I)
      -\sum_{\emptyset\neq J\subseteq I}(-1)^{|J|}\mu_v(\cap_{i\in J}\Pi^i_-)
    = \mu_v(\Sect^I)
      -\sum_{\emptyset\neq J\subseteq I}(-1)^{|J|}\mu_v(\cap_{i\in J}\Pi^i_+)\\
    &= \theta_v^I(\eta)
        -\sum_{\emptyset\neq J\subseteq I}(-1)^{|J|}\theta_v^J(\eta).      
  \end{align*}
  Now, observing that $1=\theta_v^\emptyset(\eta)$, and that $\theta_v^I$ cancels with $\theta_v^J$ when $J=I$ (because $|I|$ is  even) we get
  $$0=\sum_{J\subsetneq I}(-1)^{|J|}\theta_v^J(\eta)$$
  where the sum includes the case $|J|=\emptyset$, and it holds for all $v\notin I$.
  Now, for any fixed odd $k$, we can sum over all $I$ of cardinality $d-k+1$, and over all vertices of $\Delta^d$, getting
  $$0=\sum_v\sum_{\substack{|I|=d-k+1\\v\notin I}}\sum_{J\subsetneq I}(-1)^{|J|}\theta_v^J(\eta).$$
  By rearranging terms we get

  \begin{align*}
  0&= \sum_{\substack{J\\0\leq|J|\leq d-k}}(-1)^{|J|}\sum_{\substack{I\supset
      J\\|I|=d-k+1}}\sum_{v\notin I}\theta_v^J(\eta)\\
  &=
    \sum_{\substack{J\\0\leq|J|\leq d-k}}(-1)^{|J|}
    \sum_{v\notin J}\theta_v^J(\eta)
    \sharp\{I\supset J, |I|=d-k+1, v\notin I\}.
  \end{align*}

  It is readly checked that the number of $I\supset J$ with cardinality $d-k+1$ and not containing $v$ is 
  $$
    \sharp\{I\supset J, |I|=d-k+1, v\notin I\}=\binom{d-|J|}{d-|J|-k+1}=\binom{d-|J|}{k-1}.
  $$
  Therefore
  \begin{align*}
    0&=\sum_{\substack{J\\0\leq|J|\leq d-k}}(-1)^{|J|}
    \sum_{v\notin J}\theta_v^J(\eta)\binom{d-|J|}{k-1}
    =\sum_{\substack{J\\0\leq|J|\leq d-k}}(-1)^{|J|}
    \binom{d-|J|}{k-1}\sum_{v\notin J}\theta_v^J(\eta)\\
    &=
  \sum_{\substack{J\\0\leq|J|\leq d-k}}(-1)^{|J|}(d+1-|J|)\binom{d-|J|}{k-1}
    \frac{1}{d+1-|J|}\sum_{v\notin J}\theta_v^J(\eta)\\
    &=\sum_{\substack{J\\0\leq|J|\leq d-k}}(-1)^{|J|}(d+1-|J|)\binom{d-|J|}{k-1}
    \bar\theta^J(\eta)\\
    &=k\sum_{\substack{J\\0\leq|J|\leq d-k}}(-1)^{|J|}\binom{d+1-|J|}{k}
    \bar\theta^J(\eta)
  \end{align*}
  from which we deduce
  \begin{align*}
    0&=\sum_{\substack{J\\0\leq|J|\leq d-k}}(-1)^{|J|}\binom{d+1-|J|}{k}
    \bar\theta^J(\eta)\\
     &=\sum_{j=0}^{d-k}(-1)^j\binom{d+1-j}{k}\sum_{|J|=j}\bar\theta^J(\eta)=
       \sum_{j=0}^{d-k}(-1)^j\binom{d+1-j}{k}\alpha^{j}
  \end{align*}
  now, by setting $d+1-j=k+i$ (hence $j=d+1-k-i$ and $i=d+1-j-k$) we get
  $$0=\sum_{i=1}^{d-k+1}(-1)^{d+1-k-i}\alpha^{d+1-k-i}\binom{k+i}{k}=
  \sum_{i=1}^{d-k+1}(-1)^i\alpha_{k-1+i}\binom{k+i}{k}$$
\end{proof}

This concludes the proof of \cref{teo:Thetaeven}.
\end{proof}

Let us now specialise to the case of  convex projective manifold $M=\Omega/\Gamma$. In this case one can also consider ideal simplices in $\Omega$. Generically, angles at ideal vertices will vanishes.
In the $4$-dimensional case, if we denote by $\theta_{F}(S):=\bar\theta^I(S)$ the angle around a $2$-dimensional face $F=S^I$ of a simplex $S$, when $S$ tends to be ideal we get

\begin{equation}\label{eq:thu}
    \wh\bTheta(S)=\sgn(S) \frac{1}{2}\left(2-\sum_{\dim(F)=2}\theta_{F}(S)\right).
\end{equation}

In the real hyperbolic case, this is  up to a uniform constant the so called Thurston's formula for the volume of ideal geodesic $4$-simplices in hyperbolic $4$-spaces\footnote{The Thurston formula computes the volume of a geodesic ideal $4$-simplex $P$ as $vol(P)=\frac{\pi}{3}(4\pi-\sum(\text{dihedral angles}))$. See for instance~\cite{MilnorVol1}.} (that can be derived from the Schläfli's differential Formula, observing that hyperbolic ideal triangles have all the same area). 

Equation \ref{eq:thu} could potentially be used to prove that convex projective $4$-manifolds $M$ satisfy the inequality $\chi(M)\geq 0$, which would be a particular case of the so called Hopf-Thurston conjecture\footnote{That predicts that any closed aspherical $2n$-manifold satisfies $(-1)^n\chi(M)\geq 0$, and which is still an open question in full generality, in particular it is not known for convex projective manifolds.}.
Even further, one might possibly show the strict positivity $\chi(M)>0$ in certain convex domains, which would give a positive answer to \cref{question:poschar} and \cref{question:posvol}.

We know that this cannot be the case in general. For instance, there are example where $M$ fibers in circles (in this case $\Omega$ is decomposable) so $\chi(M)=0$. In the Euclidean case\footnote{We notice that it is not difficult to show that in $\R^n$ with any constant Finsler norm $||\cdot||$, the cochain $\bar\bTheta$ vanishes.} the cocycle $ \wh\bTheta$ vanishes identically. Now by \cite{MitulRago} we know that in convex domains $\Omega\subseteq \R\Pro^d$ which are not politopes, the top-dimensional straight simplices are {\em  slim}, meaning that there is a uniform constant $\delta$ such that any face is contained in the $\delta$-neighbourhood of the others. This is a manifestation of some negative curvature behaviour, and suggests that dihedral angles should be smaller than in the Euclidean case, giving concrete hope to understand in which case we have $\wh\Theta(S)>0$.

\subsection{A combinatorial version of the Pfaffian of curvature}
The fact that odd dimensional terms disappeared from \cref{eq:volume} is not a coincidence. In fact, the generalized Chern-Gauss-Bonnet Formula via the Pfaffian of the curvature  produces an expression for computing the Euler characteristic (see for instance \cite{McMullen-GB} and \cite{KimWan}). That formula is not gentle, but one realises that odd-dimensional terms depend on the second fundamental form, so they disappear in case we have a triangulation with totally geodesic faces. In dimension four  and constant curvature that formula specializes to Equation (\ref{eq:volume}).  In general it remains tractable, which is indeed the approach used in \cite{KimWan}.

We remark that the Chern-Gauss-Bonnet formula can be given for Finlser manifolds (see~\cite{BaoChern}).

Our approach can be considered a combinatorial version of the Chern-Gauss-Bonnet theorem. Let $M$ be a $d$-manifold equipped with a Finsler structure.
Consider a simplex $S\in\Comb_d(M)$. Suppose that as the diameter of $S$ goes to zero, then the ratio
\begin{equation}
  \label{eq:omega}
\frac{\bar\bTheta(S)}{\op{diam}(S)^d}
\end{equation}
is bounded and has a limit. Then one could define a $d$-form $\omega$ as follows. Given $v_1,\dots,v_d\in T_xM$ consider simplices $S_n$ having the first simplex in $x$ and the others on geodesics emanating from $x$, tangent to $v_i$, at distance $1/n$ from $x$. Then define $$\omega(v_1\wedge\dots\wedge v_d)=\lim_{n\to\infty}\bar\bTheta(S_n)n^d.$$

This form $\omega$ satisfies $\int_M\omega=\chi(M)$ because $\bar\bTheta$ computes $\chi(M)$ on arbitrarily fine refinements of triangulations.
If in addition one could show that $\bar\bTheta$ is a cocyle (see \cref{sec:cocornot}), then one would get that for any subdivision of a simplex $\eta$, $\bar\bTheta(\eta)$ is the sum over the contributions of the simplices of the subdivision. Making arbitrarily fine subdivision, one would get that for any immersed simplex $\eta$ one would have $$\bar\bTheta(\eta)=\int_\eta\omega.$$

Again, if $\bar\bTheta$ is not a cocycle, we can $\e$-average it, obtaining an averaged version $\omega_\e$ of $\omega$ which is a differential cocycle that computes $\chi(M)$.

Finally, we observe that if $\omega$ is a pointwise positive (negative) multiple of the volume form, then one could deduce that $\chi(M)$ is positive (negative).

\section{Straightenings with geodesic $1$-skeleton are generically not immersed}
\label{examples}

Let $\Omega$ be a simply connected nonpositively curved  oriented  Riemannian manifold. A straightening
$$\sigma:\Omega^*\longrightarrow C_*(\Omega)$$
is a map $\sigma: \Omega^k\rightarrow C_k(\Omega)$, for every $k\geq 0$, satisfying the properties (1-4), (6) of our Definition \ref{def:Mais} of immersed straightenings. 

\begin{definition} A straightening is said to be \emph{geodesic} if $\sigma(x,y):\Delta^1\rightarrow \Omega $ is the constant map $\equiv x$ when $x=y$, and the geodesic segment from $x$ to $y$ parametrized by arc length if $x\neq y$. 
\end{definition}

Most natural straightenings present in the litterature are geodesic. This is the case for the straightening obtained by inductive geodesic coning, or the barycenter straightening defined, for $(x_0,\dots,x_k)$, as sending $(t_0,\dots,t_k)$ to the unique minimizer of the function $\Omega \ni x\mapsto \sum_{i=0}^k t_i d(x,x_i)^2$. 

The point of this section is to show that in the absence of totally geodesic submanifolds, which typically happens in non constant curvature, geodesic straightenings are never immersed. First, we define a smooth map 
$$\begin{array}{rcl}
\Omega \times \Omega&\longrightarrow & T\Omega\\
(x,y)&\longmapsto &V_x^y\in T_x\Omega,
\end{array}$$
where $V_x^y$ is defined to be $0$ if $x=y$, and tangent to the geodesic from $x$ to $y$ with $\|V_x^y\|=d(x,y)$ otherwise. 

\begin{theorem} \label{thm not immersed 1} Suppose that there exists $x_1,\dots,x_d\in \Omega $ not contained in any totally geodesic $(d-1)$-submanifold of $\Omega$, satisfying 
$$\mathrm{Dim}(V_{x_d}^{x_1},\dots, V_{x_d}^{x_{d-1}})=d-1.$$ 
Then there exists an open set $U$ in $\Omega^{d+1}$ such that for any geodesic straightening on $\Omega$, the singular simplex $\sigma(y_0,\dots,y_d)$  is not immersed for any $(y_0,\dots,y_d)\in U$. 
\end{theorem}

Theorem \ref{thm not immersed 1} applies in particular to complex hyperbolic spaces $\mathbb{H}^n_\mathbb{C}$ which contain no totally geodesic $(2n-1)$-submanifolds, or products of real hyperbolic spaces, the simplest being $\mathbb{H}^2_\mathbb{R}\times \mathbb{H}^2_\mathbb{R}$. The latter products do contain co-dimension $1$ totally geodesic submanifolds, but generic tuples of points will not be contained in these totally geodesic submanifolds. 

We will prove a much stronger version of Theorem  \ref{thm not immersed 1} by showing that the failure to be immersive can be detected already from the differential restricted at the vertices of the singular simplices. We start by introducing some terminology.

For every $(x_0,\dots,x_d )\in \Omega^{d+1}$ we define 
$$\mathrm{Or}_{x_0}(x_1,\dots,x_d)=\mathrm{Or}(V_{x_0}^{x_1},\dots,V_{x_0}^{x_d})\in \{\pm 1, 0\},$$
where the orientation of the $d$ vectors $V_{x_0}^{x_1},\dots,V_{x_0}^{x_d}\in T_{x_0}\Omega$ is $0$ when they span a subvector space of dimension $< d$. Now observe that if $\sigma:\Delta^d\rightarrow \Omega$ is a geodesic straightening with vertices $x_0,\dots,x_d$, then 
\begin{equation}\label{failure imm}
0\neq (-1)^i\mathrm{Or}_{x_i}(x_0,\dots,\widehat{x_i},\dots,x_d)=(-1)^j\mathrm{Or}_{x_j}(x_0,\dots,\widehat{x_j},\dots,x_d),
\end{equation}
for any $i\neq j$. Our approach is based on the obvious observation that any geodesic straightening violating this relation will not be immersed. 

\begin{theorem}  \label{thm not immersed 2} Suppose that there exists $x_1,\dots,x_d\in \Omega $ not contained in any totally geodesic $(d-1)$-submanifold of $\Omega$, satisfying 
$$\mathrm{Dim}(V_{x_d}^{x_1},\dots, V_{x_d}^{x_{d-1}})=d-1.$$ 
Then there exists a $(d+1)$-tuple $(y_0,\dots,y_d)\in \Omega^{d+1}$ such that 
$$\mathrm{Or}_{y_0}(y_1,y_2,\dots,y_d)\neq 0$$
and 
$$\mathrm{Or}_{y_d}(y_0,y_1\dots,y_{d-1})=0.$$
\end{theorem}

\begin{proof}[Theorem  \ref{thm not immersed 2} $\Rightarrow$ Theorem  \ref{thm not immersed 1}] If $\sigma(y_0,\dots,y_d)$ is an immersed simplex, then the $(d+1)$-tuple of points $(y_0,\dots,y_d)$ satisfies the relation (\ref{failure imm}). Since the condition that $\mathrm{Or}_{y_i}(y_0,\dots,\widehat{y_i},\dots,y_d)=\pm 1$ is an open condition on $\Omega^{d+1}$, it is enough to exhibit one $(d+1)$ tuple $(y_0,\dots,y_d)$ such that there exists $i\neq j$ satisfying $\mathrm{Or}_{y_i}(y_0,\dots,\widehat{y_i},\dots,y_d)\neq 0$ and $\mathrm{Or}_{x_j}(x_0,\dots,\widehat{x_j},\dots,x_d)\neq 0$, and so that  the relation (\ref{failure imm}) fails for these $i\neq j$. 

Theorem  \ref{thm not immersed 2} ensures the existence of a $(d+1)$-tuple $(y_0,\dots,y_d)\in \Omega^{d+1}$ such that 
$$\mathrm{Or}_{y_0}(y_1,y_2,\dots,y_d)\neq 0$$
and 
$$\mathrm{Or}_{y_d}(y_0,y_1\dots,y_{d-1})=0,$$
which is not quite what we need (taking $i=0$ and $j=d$). But since the condition $\mathrm{Or}_{y_0}(y_1,y_2,\dots,y_d)\neq 0$ is an open condition, we can perturb $y_1,\dots,y_{d-1}$ without changing the sign of this orientation. At the same time, perturbing $y_1,\dots,y_{d-1}$ we can get both $\mathrm{Or}_{y_d}(y_0,y_1\dots,y_{d-1})=1$ or $\mathrm{Or}_{y_d}(y_0,y_1\dots,y_{d-1})=-1$, thus there is at least one choice violating the relation  (\ref{failure imm}). 
\end{proof}

\begin{proof}[Proof of Theorem \ref{thm not immersed 2}] Suppose that $x_1,\dots,x_d\in \Omega$ are not contained in any $(d-1)$-dimensional totally geodesic submanifold of $\Omega$.  We consider 
$$V(x_1,\dots,x_d):=\{x\in \Omega \mid \mathrm{Or}_{x}(x_1,\dots,x_d)=0\}.$$
This is a submanifold of $\Omega$ since $\mathrm{Or}_{x}(x_1,\dots,x_d)=0$ if and only if $V_x^{x_1}\wedge \dots \wedge V_x^{x_d}=0$. The condition $\mathrm{Dim}(V_{x_d}^{x_1},\dots, V_{x_d}^{x_{d-1}})=d-1$ ensures that $V(x_1,\dots,x_d)$ is $(d-1)$-dimensional.  

If there exists $x\in V(x_1,\dots,x_d)$ such that $\mathrm{Or}_{x_d}(x,x_1,\dots,x_{d-1})\neq 0$ then we are done since we can take $(x,x_1,\dots,x_d)$. Suppose from now on that there are no such $x\in  V(x_1,\dots,x_d)$. Note that in particular since the span of $V_{x_d}^{x_1},\ldots,V_{x_d}^{x_{d-1}}$ is $(d-1)$-dimensional we necessarily have $V_{x_d}^x\in \langle V_{x_d}^{x_1},\ldots,V_{x_d}^{x_d}\rangle$ for every $x\in V(x_1,\dots,x_d)$. 

Since $V(x_1,\dots,x_d)$ contains $x_1,\dots,x_d$, it cannot be totally geodesic. Let $x,y\in V(x_1,\dots,x_d)$ be such that the geodesic segment between $x$ and $y$ is not included in $V(x_1,\dots,x_d)$. 

We claim that we can choose $x,y$ such that $V_x^y\notin \langle V_x^{x_1},\dots,V_x^{x_{d-1}}\rangle $. Indeed, if any geodesic segment joining two points in $V(x_1,\dots,x_d)$ would be tangent to the initial endpoint then $V(x_1,\dots,x_d)$ would be totally geodesic. 

Finally for $x,y$ as in the claim we can now take the $(d+1)$-tuple to be 
$$(x,y,x_2,\dots,x_d).$$ 
Indeed, the claim ensures that the dimension of the span of $V_x^y,V_x^{x_1},\dots,V_x^{x_{d-1}}$ is $d$ so that
$$ \mathrm{Or}_x(y,x_2,\dots,x_d)\neq 0.$$
Finally since $V_{x_d}^x$ and $V_{x_d}^y$ both lie in the $(d-1)$-dimensional space $\langle V_{x_d}^{x_1},\ldots,V_{x_d}^{x_{d-1}}\rangle $, we can add them both to the span without increasing the dimension. We conclude that 
$$\mathrm{Dim}\langle V_{x_n}^x,V_{x_n}^y, V_{x_n}^{x_2},\dots,V_{x_d}^{x_{d-1}}\rangle \leq \mathrm{Dim}\langle V_{x_n}^x,V_{x_n}^y, V_{x_n}^{x_1},V_{x_n}^{x_2},\dots,V_{x_d}^{x_{d-1}}\rangle =d-1,$$
so that $\mathrm{Or}_x(y,x_2,\dots,x_d)\neq 0$, which concludes the proof. \end{proof}

Note also that in contrast to what is sometimes claimed in the litterature, the straightening by the inductive coning procedure is in general not smooth. Indeed if the inverse image through the exponential map $\mathrm{exp}_{x_0}:\mathbb{R}^d\rightarrow \Omega$ of a straightening $\sigma(x_1,\dots,x_d)$ is not included in the linear span of $\mathrm{exp}_{x_0}^{-1}(x_1),\dots,\mathrm{exp}_{x_0}^{-1}(x_d)$, then the simplex $\sigma(x_0,x_1,\dots,x_d)$ obtained by coning $\sigma(x_1,\dots,x_d)$ along geodesic segments from $x_0$ will not be smooth at $x_0$. This is however a minor issue in practice, since an appropriate reparametrization renders a smooth singular simplex (with singular differential at $x_0$).

\bigskip

\textbf{Tools use:} In the writing process, we mildly used AI for the purpose of typesetting (e.g. GitHub Copilot for \LaTeX\ coding pictures) and identification of minor errors. All the ideas and proofs are from the authors.

\bigskip 

\bibliographystyle{plain}
\bibliography{EulerCharBiblio}

\end{document}